\documentclass[11pt,a4paper]{amsart}
\usepackage{color}
\usepackage{textcomp}
\usepackage{mathtools}
\usepackage{amstext}
\usepackage{amsthm}
\usepackage{amssymb}
\usepackage{setspace}
\usepackage{fullpage}
\makeatletter
\numberwithin{equation}{section}
\numberwithin{figure}{section}
\theoremstyle{plain}
\newtheorem{thm}{\protect\theoremname}[section]
  \theoremstyle{definition}
  \newtheorem{defn}[thm]{\protect\definitionname}
  \theoremstyle{remark}
  \newtheorem{rem}[thm]{\protect\remarkname}
  \theoremstyle{plain}
  \newtheorem{prop}[thm]{\protect\propositionname}
  \theoremstyle{definition}
  \newtheorem{example}[thm]{\protect\examplename}
  \theoremstyle{plain}
  \newtheorem{conjecture}[thm]{\protect\conjecturename}
  \theoremstyle{plain}
  \newtheorem{cor}[thm]{\protect\corollaryname}
  \theoremstyle{plain}
  \newtheorem{lem}[thm]{\protect\lemmaname}
  \theoremstyle{remark}
  \newtheorem{notation}[thm]{\protect\notationname}

\makeatother

  \providecommand{\conjecturename}{Conjecture}
  \providecommand{\corollaryname}{Corollary}
  \providecommand{\definitionname}{Definition}
  \providecommand{\examplename}{Example}
  \providecommand{\lemmaname}{Lemma}
  \providecommand{\notationname}{Notation}
  \providecommand{\propositionname}{Proposition}
  \providecommand{\remarkname}{Remark}
\providecommand{\theoremname}{Theorem}

\def\Qp{\mathbb Q_p}
\def\Zp{\mathbb{Z}_p}

\def\val{\mathrm{val}}
\def\ac{\mathrm{ac}}

\def\complex{\mathbb{C}}
\def\reals{\mathbb{R}}
\def\nats{\mathbb{N}}
\def\ints{\mathbb{Z}}
\def\rats{\mathbb{Q}}
\def\supp{\mathrm{supp}}

\def\VF{\mathrm{VF}}
\def\RF{\mathrm{RF}}
\def\VG{\mathrm{VG}}

\def\spec{\mathrm{Spec}}
\def\Ldp{\mathcal{L}_\mathrm{DP}}

\newcommand\curly[1]{\{ #1 \}}

\AtEndDocument{\bigskip{\footnotesize%
  \textsc{Itay Glazer, Faculty of Mathematics and Computer Science, The  Weizmann Institute of
Science,  Rehovot 76100, ISRAEL.
} \\
  \textit{E-mail address}: \texttt{itayglazer@gmail.com} \par
  \addvspace{\medskipamount}
}}
\AtEndDocument{\bigskip{\footnotesize
  \textsc{Yotam I. Hendel, Department of Mathematics, Northwestern University,  2033 Sheridan Road, Evanston, IL 60208, USA.
} \\
  \textit{E-mail address}: \texttt{yotam.hendel@gmail.com} \par
  \addvspace{\medskipamount}
}}

\begin{document}
\author{Itay Glazer and Yotam I. Hendel}
\title{On singularity properties of convolutions of algebraic morphisms}
\maketitle
\begin{abstract}
Let $K$ be a field of characteristic zero, $X$ and $Y$ be smooth
$K$-varieties, and let $V$ be a finite dimensional $K$-vector space.
For two algebraic morphisms $\varphi:X\rightarrow V$ and $\psi:Y\rightarrow V$
we define a convolution operation, $\varphi*\psi:X\times Y\to V$,
by $\varphi*\psi(x,y)=\varphi(x)+\psi(y)$. We then study the singularity
properties of the resulting morphism, and show that as in the case
of convolution in analysis, it has improved smoothness properties.

Explicitly, we show that for any morphism $\varphi:X\rightarrow V$
which is dominant when restricted to each irreducible component of
$X$, there exists $N\in\mathbb{N}$ such that for any $n>N$ the
$n$-th convolution power $\varphi^{n}:=\varphi*\dots*\varphi$ is
a flat morphism with reduced geometric fibers of rational singularities
(this property is abbreviated (FRS)). By a theorem of Aizenbud and Avni,
for $K=\mathbb{Q}$, this is equivalent to good asymptotic behavior
of the size of the $\mathbb{Z}/p^{k}\mathbb{Z}$-fibers of $\varphi^{n}$
when ranging over both $p$ and $k$.

More generally, we show that given a family of morphisms $\{\varphi_{i}:X_{i}\rightarrow V\}$
of complexity $D\in\mathbb{N}$ (i.e. that the number of variables
and the degrees of the polynomials defining $X_{i}$ and $\varphi_{i}$
are bounded by $D$), there exists $N(D)\in\mathbb{N}$ such that
for any $n>N(D)$, the morphism $\varphi_{1}*\dots*\varphi_{n}$ is
(FRS).
 
\end{abstract}

\tableofcontents{}

\section{Introduction}

\subsection{Motivation}

Given functions $f,g\in L^{1}(\reals^{n})$, their convolution is
defined by $(f\ast g)(x)=\int_{\reals^{n}}f(x-t)g(t)dt$. It is a  basic fact  
 that $f*g$ has better smoothness properties than both
$f$ and $g$; 
indeed, if in addition we assume that $f\in{C}^{k}(\reals^{n})$ and
$g\in{C}^{l}(\reals^{n})$, then $(f*g)'=f'*g=f*g'$ and therefore
$f*g\in C^{k+l}(\reals^{n})$. In particular, if either $f$ or $g$
is smooth then $f*g$ is smooth. An interesting question is whether
this phenomenon has an analogue in the setting of algebraic geometry.
In this paper we give a partial answer to this question
by exploring the following operation, as proposed by A. Aizenbud and
N. Avni: 
\begin{defn}
\label{def:convolution} Let $X$ and $Y$ be algebraic varieties,
$G$ an algebraic group and let $\varphi:X\to G$ and $\psi:Y\to G$
be algebraic morphisms. We define their convolution $\varphi*\psi:X\times Y\to G$
by $\varphi*\psi(x,y)=\varphi(x)\cdot\psi(y)$. In particular, the
$n$-th convolution power of $\varphi$ is $\varphi^{n}(x_{1},\ldots,x_{n})=\varphi(x_{1})\cdot\ldots\cdot\varphi(x_{n})$. 
\end{defn}

\begin{rem}
\label{remark 1.2} This operation is related to the classical notion
of convolution in the following way. 
Given a finite ring $A$, the above morphisms $\psi$ and $\varphi$
induce maps $\varphi_{A}:X(A)\rightarrow G(A)$ and $\psi_{A}:Y(A)\rightarrow G(A)$.
By defining $F_{\varphi_{A}}(t):=\left|\varphi_{A}^{-1}(t)\right|$,
we have,
\[
F_{\varphi_{A}\ast\psi_{A}}(t)=\sum_{s\in G(A)}F_{\varphi_{A}}(s)\cdot F_{\psi_{A}}(s^{-1}t)=F_{\varphi_{A}}*F_{\psi_{A}}(t).
\]
\end{rem}

Recall that a morphism $\varphi:X\to Y$ between $K$-varieties is
smooth if it is flat and the fiber $\varphi^{-1}\circ\varphi(x)$
is smooth for all $x\in X(\overline{K})$. Analogously to the analytic
situation, convolving a smooth morphism with any morphism
 whose domain is a smooth variety yields a smooth morphism. More generally,
the convolution operation preserves properties of morphisms which
are stable under base change and compositions. 
\begin{prop}[see Proposition \ref{prop:convolution preserves} for a proof]
Let $X$ and $Y$ be varieties over a field $K$, let $G$ be an
algebraic group over $K$ and let $S$ be a property of morphisms
that is preserved under base change and compositions. If $\varphi:X\to G$
is a morphism that satisfies the property $S$, the natural map $i_{K}:Y\to\mathrm{Spec}(K)$
has property $S$ and $\psi:Y\to G$ is arbitrary, then $\varphi*\psi$
has property $S$. 
\end{prop}

A natural question is then whether any morphism $\varphi:X\to G$
becomes smooth after sufficiently many convolution powers. Clearly,
we must first ensure that $\varphi$ becomes dominant when raised
to high enough convolution power, but even if this is indeed the case,
we might have that $\varphi^{n}$ is not smooth for all $n\in\nats$,
as shown in the next example. 
\begin{example}
\label{Example- non-smooth morphism} Consider the map $\varphi(x)=x^{2}$.
Then $d\varphi_{(0,\ldots,0)}^{n}$ is not surjective for every $n\in\nats$,
and thus $\varphi^{n}$ is not smooth for all $n$. 
\end{example}

Although in Example \ref{Example- non-smooth morphism} above $\varphi^{n}$
is not smooth at $(0,\ldots,0)$ for all $n\in\nats$, it has better
singularity properties as $n$ grows. While the singular locus of
$\varphi$ is of codimension $1$ and its fiber over $0$ is non-reduced,
$\varphi^{2}$ has reduced fibers and its singular locus is of codimension
$2$. For $n=3$ we get that the singular locus of $\varphi^{3}$
is of codimension $3$, and that the  non-smooth fiber $(\varphi^3)^{-1}(0)$
is reduced and has rational singularities (for rational singularities
see Definition \ref{def:ratlsings}(2)).
In the case where $G=(\mathbb{A}^1,+)$ as above, this 
is analogous to 
the following Thom-Sebastiani type result; if $f_1,f_2 \in \complex[x_1,\ldots,x_{n}]$, then 
 $\mathrm{lct}(f_1*f_2)=\min\{1,\mathrm{lct}(f_1)+\mathrm{lct}(f_2)\}$, 
 where $\mathrm{lct}(h)$ denotes the log-canonical threshold of the hypersurface defined by $h \in \complex[x_1,\ldots,x_t]$ in $ \complex^t$ 
(e.g.~\cite[Corollary 1]{MSS}, see also \cite{Mu12}). 

This motivates us to introduce the (FRS) property of morphisms as defined in \cite{AA16} (see Subsection
\ref{subsection: Further discussion} for further discussion of the
(FRS) property). From now on, we assume that $K$ is a field of characteristic
zero. 
\begin{defn}
\label{def:(FRS)} Let $X$ and $Y$ be smooth $K$-varieties. We
say that a morphism $\varphi:X\rightarrow Y$ is (FRS) if it is flat
and if every geometric fiber of $\varphi$ is reduced and has rational
singularities. 
\end{defn}

This property plays a key role in this paper, and as seen in the above
example, although we cannot hope to obtain smooth morphisms after
sufficiently many convolution powers, we might be able to get morphisms
with the (FRS) property. It is conjectured by Aizenbud and Avni that
such a phenomenon occurs, under adequate assumptions, for a general
algebraic group $G$: 
\begin{conjecture}
\label{main conjecture} Let $X$ be a smooth, absolutely irreducible
$K$-variety, $G$ be an algebraic $K$-group and let $\varphi:X\to G$
be a morphism such that $\varphi(X)\not\subseteq gH$ for any translation
of an algebraic subgroup $H\leq G$ by an element $g\in G(\overline{K})$.
Then there exists $N\in\mathbb{N}$ such that for any $n>N$, the
$n$-th convolution power $\varphi^{n}$ is (FRS). 
\end{conjecture}

\subsection{Main results}

In this paper we verify Conjecture \ref{main conjecture} for the
case of a vector group. Namely, we prove the following theorem: 
\begin{thm}
\label{Main result} Let $K$ be a field of characteristic zero. Let $X$ be a smooth $K$-variety 
with absolutely
irreducible components $X_{1},\ldots,X_{l}$, 
 and
let $V$ be a $K$-vector space of finite dimension. Then for every
morphism $\varphi:X\to V$ such that $\varphi(X_i)$ is not
contained in any proper affine subspace of $V$ for all $i$, there exists $N\in\mathbb{N}$
such that for any $n>N$ the $n$-th convolution power $\varphi^{n}$
is (FRS). 
\end{thm}

In fact, we are able to prove a uniform analogue of Theorem \ref{Main result}
for families of algebraic morphisms. 
\begin{defn}
We call a morphism $\varphi:X\to Y$ \textit{strongly
dominant} if $\varphi$ is dominant when restricted to each absolutely
irreducible component of $X$. 
\end{defn}

\begin{thm}
\label{generalization of the main result} Let $K$ and $V$ be as
in Theorem \ref{Main result}, let $Y$ be a $K$-variety, let $\widetilde{X}$
be a family of varieties over $Y$, and let $\widetilde{\varphi}:\widetilde{X}\rightarrow V\times Y$
be a $Y$-morphism. Denote by $\widetilde{\varphi}_{y}:\widetilde{X}_{y}\rightarrow V$
the fiber of $\widetilde{\varphi}$ at $y\in Y$. Then,
\begin{enumerate}
\item The set $Y':=\{y\in Y:\widetilde{X}_{y}\text{ is smooth and }\widetilde{\varphi}_{y}:\widetilde{X}_{y}\rightarrow V\text{ is strongly dominant}\}$
is constructible.
\item There exists $N\in\nats$ such that for any $n>N$, and any $n$ points
$y_{1},\ldots,y_{n}\in Y'(K)$, the morphism $\widetilde{\varphi}_{y_{1}}*\dots*\widetilde{\varphi}_{y_{n}}:\widetilde{X}_{y_{1}}\times\dots\times\widetilde{X}_{y_{n}}\rightarrow V$
is (FRS). 
\end{enumerate}
\end{thm}

\begin{defn}
\label{def:complexity}An affine variety $X$ has \emph{complexity
at most} $D\in\nats$, if $X$ can be written as $K[X]=K[x_{1},\ldots,x_{m}]/\langle f_{1},\ldots,f_{k}\rangle$,
where $m,k$ and the maximal degree of all polynomials, $\max\limits _{i}\{f_{i}\}$,
is at most $D$. The notion of complexity can be similarly defined
for non-affine varieties and for morphisms (see Section \ref{sec:Convolution-properties-of}). 
\end{defn}

As a corollary of Theorem \ref{generalization of the main result},
one can then show that given a complexity class $D\in\nats$, there
exists $N(D)\in\nats$ such that the convolution of any $n>N(D)$
morphisms of complexity at most $D$ is (FRS). 
\begin{cor}
\label{Generalization for varieties of bounded complexity}Let $V$
be a $K$-vector space of dimension not greater than $D \in \nats$. Then 
there exists an integer $N(D)\in\nats$ such that for every $n>N(D)$
the morphism $\varphi_{1}*\dots*\varphi_{n}$ is (FRS) for any $n$
strongly dominant morphisms $\varphi_{i}:X_{i}\rightarrow V$ of complexity
at most $D$, from smooth $K$-varieties $X_{i}$ to $V$. 
\end{cor}

\subsection{A brief discussion of the (FRS) property}

\label{subsection: Further discussion} The notion of rational singularities
can be regarded as a certain approximation of smoothness of varieties.
Thus, we can view the (FRS) property, which is its relative analogue,
as an approximation to smoothness of morphisms. 
In this sense, Theorem \ref{Main result} supports the claim that
the convolution operation improves the singularity properties of morphisms.
To 
explain our particular interest in the (FRS) property, we present it from several
different points of view: 
\begin{enumerate}
\item \textbf{The number-theoretic point of view:} Let $X$ be a finite
type $\ints$-scheme such that $X_{\rats}=X\times_{\spec(\ints)}\spec(\rats)$
is a local complete intersection. By \cite[Theorem 3.0.3]{AA18} and
\cite[Theorem 1.3]{Gla19}, which are based on works of Denef \cite{Den87}, Igusa \cite{Igu00} and Mustata \cite{Mus01}, it turns out that 
$X_{\rats}$
has rational singularities 
if and only if there exists a positive constant $C\in\reals$ such that for any
prime $p$ and any $k\in\mathbb{N}$ we have
\begin{align}
\frac{\left|X(\ints/p^{k}\ints)\right|}{p^{k\cdot\dim(X{}_{\rats})}}<C.\label{formula:asymptotic point count}
\end{align}
Now, given a $\mathbb{\ints}$-morphism $\varphi:X\to Y$,
between finite type, reduced $\mathbb{Z}$-schemes such that $\varphi_{\mathbb{Q}}:=\varphi\times_{\spec(\ints)}\spec(\rats)$
is (FRS), then for any element $y\in Y(\mathbb{Z})$, the fiber $X_{y}$
satisfies that $\left(X_{y}\right)_{\rats}$ is a reduced, local complete
intersection variety with rational singularities. By the above arguments
we see that the size of the fiber $X_{y}(\ints/p^{k}\ints)$,
where $y\in Y(\ints)$, 
behaves asymptotically as if $\varphi$ were smooth. 
Combining the above with Remark \ref{remark 1.2} allows one to interpret
the (FRS) property from a probabilistic point of view. 
\item \textbf{The probabilistic point of view:} Let $X$ be a smooth finite
type $\ints$-scheme, let $G$ be an algebraic group over $\ints$,
let $\varphi:X\rightarrow G$ be a morphism, and let $A$ be a finite
ring. Recall we have defined $F_{\varphi_{A}}(t)=|\varphi_{A}^{-1}(t)|$.
We then saw that taking the size of the fiber commutes with convolution,
that is $F_{\varphi_{A}^{n}}=F_{\varphi_{A}}* \ldots *F_{\varphi_{A}} $. Now, taking the
uniform probability measure $1_{X(A)}$ on $X(A)$ gives rise to a
random walk on $G(A)$ with probability distribution $\varphi_{*}(1_{X(A)})$,
whose $m$-th step has the following probability distribution: 
\[
\varphi_{*}(1_{X(A)})^{m}=\varphi_{*}^{m}(1_{X(A)}\times\dots\times1_{X(A)})=\frac{F_{\varphi_{A}^{m}}}{\left|X(A)\right|^{m}}.
\]
Thus, for $A=\ints/p^{k}\ints$, existence of $m$ large enough such
that $\varphi_{\rats}^{m}$ is (FRS) would imply by Formula \ref{formula:asymptotic point count}
that after $m$ steps the probability distribution of the random walk
is not too far from being uniform on $G(A)$. Consequentially, verifying
Conjecture \ref{main conjecture} leads to interesting uniform results
on random walks on families of finite groups and compact $p$-adic groups. Further work in this
direction is in progress, and will appear in upcoming papers. 
\item \textbf{The analytic point of view:} There is a useful analytic characterization
of the (FRS) property of a morphism $\varphi:X\to Y$
defined over a non-Archimedean local field $F$ of characteristic
$0$. Explicitly, $\varphi$ is (FRS) if and only if for every locally
constant and compactly supported measure $\mu_{X}$ on $X(F)$, the pushforward
$\varphi_{*}(\mu_{X})$ has continuous density with respect to any smooth
non-vanishing measure $\mu_{Y}$ on $Y(F)$ (see Theorem \ref{Analytic condition for (FRS)}
or \cite[Theorem 3.4]{AA16}). 
Such a characterization exists also for smooth and strongly dominant
morphisms: 
\begin{enumerate}
\item If $\varphi$ is strongly dominant, then $\varphi_{*}(\mu_{X})$ is
absolutely continuous with respect to any smooth non-vanishing measure
$\mu_{Y}$ on $Y(F)$, and hence it has an ${L}^{1}$-density (\cite[Corollary 3.6]{AA16}).
\item If $\varphi$ is a smooth morphism, then $\varphi_{*}(\mu_{X})$ has
a smooth density with respect to any smooth non-vanishing measure
$\mu_{Y}$ on $Y(F)$ (see \cite[Proposition 3.5]{AA16}). 
\end{enumerate}
\end{enumerate}

The above approaches to the (FRS) property have some very interesting
applications. In \cite{AA16}, Aizenbud and Avni show that for any
semisimple algebraic group $G$ over a field $K$ of characteristic
$0$, the commutator map $\varphi:G\times G\rightarrow G$ via $\varphi(g,h)=ghg^{-1}h^{-1}$
becomes (FRS) after finitely many self-convolutions (see \cite[Theorem VIII]{AA16}).
As an application, they give a polynomial bound (in $N$) on the number
of irreducible $N$-dimensional representations of open compact subgroups
of $G(F)$, for any non-Archimedean local field $F$ of characteristic
$0$ (see \cite[Theorem A]{AA16}), and the same for arithmetic groups
of higher rank (see \cite[Theorem B]{AA18}). 

These applications and different characterizations supply us with
enough evidence that the (FRS) property encodes valuable information,
and that Conjecture \ref{main conjecture} is of interest. We believe
that the case of a vector group $G=V$ (Theorems \ref{Main result}
and \ref{Generalization for varieties of bounded complexity}) is
an important step towards proving Conjecture \ref{main conjecture}.

\subsection{\label{subsec:Sketch-of-the}Sketch of the proof of the main result
and structure of the paper}

In Section 2 we recall relevant background material: in Subsections
2.1 and 2.2 we give necessary preliminaries from model theory, and
in Subsection 2.3 we mainly review required notions from algebraic
geometry and analysis on manifolds. The definition of the (FRS) property
and the statement of the equivalent analytic criterion for the (FRS)
property also appear in Subsection 2.3. 

The scheme of the proof of the main result, Theorem \ref{Main result},
is as follows. 
Our goal is to show that when raised to high enough convolution power,
$\varphi:X\rightarrow\mathbb{A}_{K}^{m}$ satisfies the analytic criterion
for the (FRS) property as given in Theorem \ref{Analytic condition for (FRS)}.
We first prove Proposition \ref{prop:morphismdom} to reduce our problem
to the case of a strongly dominant morphism.
We then prove the theorem in the case where $K=\rats$ in the following
way. We construct a family of non-negative Schwartz measures $\{\mu_{p}\}_{p\text{ prime}}$
such that $\mu_{p}$ is a measure on $\mathbb{Q}_{p}$ and $\mathrm{\supp}(\mu_{p})=X(\Zp)$
for each prime $p$ (this is Proposition \ref{prop:existsmotivicschwartz}).
By \cite[Corollary 3.6]{AA16}, for any $p$, we deduce that $\varphi_{*}(\mu_{p})$
is a compactly supported measure which is absolutely continuous with
respect to the normalized Haar measure $\lambda$ on $\Qp^{m}$ and
consequentially its density lies in $L^{1}(\Qp^{m})$.

We now want to show there exists $n\in\nats$, such that for
any prime $p$, the pushforward under the $n$-th convolution power
$\varphi_{*}^{n}(\mu_{p}\times\dots\times\mu_{p})=\varphi_{*}(\mu_{p})*\dots*\varphi_{*}(\mu_{p})$
has continuous density with respect to $\lambda$. This implies Theorem
\ref{Main result} for $K=\rats$, since for any point $x\in(X\times\dots\times X)(\overline{\mathbb{Q}})$
there exists a prime $p$ such that $x\in\supp(\mu_{p}\times\ldots\times\mu_{p})$,
and we can then apply the analytic criterion for the (FRS) property
as given in Theorem \ref{Analytic condition for (FRS)} (see Proposition
\ref{Reduction to many Qp} for the precise statement). Two main difficulties
now arise: 
\begin{enumerate}
\item It is not true in general that every compactly supported, $L^{1}$-measure
on $\Qp^{m}$ results in a measure with continuous density after finitely
many self convolutions. This means we need to choose a measure $\mu_{p}$
that is well behaved with respect to pushing forward by $\varphi$. 
\item The required number of self convolutions needed for $\varphi_{*}(\mu_{p})$
to become (FRS) might depend on the prime $p$, while we want this
number to be independent of $p$. Thus, we need to construct a collection
of measures $\{\mu_{p}\}_{p\text{ prime}}$ that behave well in a
uniform manner. 
\end{enumerate}
We deal with these two difficulties using methods from the theory of motivic integration; in
Subsection \ref{subsec:Motivic-measures-on}, we define a notion of
a motivic measure on a $\rats$-variety, based on the notion of motivic
functions in the Denef-Pas language (see \cite{CL08,CL10,CGH14-2,CGH16},
or Subsection \ref{subsec: Denef-Pas}). The Denef-Pas language allows
us to obtain results over $\Qp$, which are uniform in $p$ for $p$
large enough, using specialization arguments (see e.g \cite[Section 4]{CGH14-2}
and Lemma \ref{specialization arguments}), thus dealing with the
second difficulty mentioned above.

To overcome the first difficulty, we show in Theorem \ref{thm:pushofconstructiblefunctionunderdominant}
that the class of motivic measures behaves well under pushforward
by algebraic morphisms. We then show that any motivic measure $\sigma$
on $\Qp^{m}$ whose density lies in $L^{1}(\Qp^{m})$ and is compactly
supported has continuous density after sufficiently many self convolutions.
We prove this by studying the decay properties of the Fourier transform
$\mathcal{F}(\sigma)$ of such motivic measures (this is Theorem \ref{thm:L^1 constructible has polynomial decaying Fourier transform}). An application of the results to $\sigma=\varphi_{*}(\mu_{p})$ finishes the proof
of Theorem \ref{Main result} for the case $K=\rats$. An
alternative proof of Theorem \ref{thm:L^1 constructible has polynomial decaying Fourier transform}, as
suggested by the anonymous referee, is given in Section
\ref{subsec:Alternative-proof-of Theorem 5.2} using results of
\cite{CGH18} on motivic exponential functions.

Finally, we show in Section \ref{sec:Reduction-to-k=00003D00003D00003D00003DQ},
that it is indeed enough to prove Theorem \ref{Main result} for $K=\rats$,
thus finishing the proof of the theorem. In Section \ref{sec:Convolution-properties-of},
we prove the relative version of Theorem \ref{Main result}, i.e Theorem
\ref{generalization of the main result} and as a consequence, obtain
Corollary \ref{Generalization for varieties of bounded complexity}.

\subsection{Conventions}

Throughout the paper we use the following conventions: 
\begin{itemize}
\item Unless explicitly stated otherwise, $K$ is a field of characteristic
$0$ and $F$ is a non-Archimedean local field of characteristic $0$
whose ring of integers is $\mathcal{O}_{F}$. 
\item For an integral subscheme $A\subseteq X$ we denote
by $K(A)$ its function field. 
\item For a morphism $\varphi:X\rightarrow Y$ of algebraic varieties, the
scheme theoretic fiber at $y\in Y$ is denoted by either $\varphi^{-1}(y)$
or $\spec(K(\curly{y}))\times_{Y}X$. 
\item For a field extension $K'/K$ and a $K$-variety $X$ (resp. $K$-morphism
$\varphi:X\rightarrow Y$), we denote the base change of $X$ (resp.
$\varphi$) by $X_{K'}:=X\times_{\spec(K)}\spec(K')$ (resp. $\varphi_{K'}:X_{K'}\rightarrow Y_{K'}$). 
\item If $X$ is a $K$-variety, we set the $N$-fold product of $X$ by
$X^{N}:=X\times\dots\times X$. 
\item If $\varphi:X\to G$ is a morphism to an algebraic group $G$ we denote
its $n$-th convolution power by $\varphi^{n}:=\varphi*\ldots*\varphi$. 
\end{itemize}

\subsection{Acknowledgements}

We thank Moshe Kamenski and Raf Cluckers for enlightening conversations
about the model theoretic settings. We thank Nir Avni for numerous
helpful discussions, as well as for proposing this problem together
with Rami Aizenbud. A large part of this work was carried out while
visiting the mathematics department at Northwestern university, we
thank them and Nir for their hospitality. Finally we wish to thank
our teacher Rami Aizenbud for answering various questions and for
helping to shape many of the ideas in this paper. We benefited from
his guidance deeply.

We also wish to thank the anonymous referees for their insightful comments and remarks, and in particular for suggesting  
the alternative proof of Theorem \ref{thm:L^1 constructible has polynomial decaying Fourier transform} (see Section \ref{subsec:Alternative-proof-of Theorem 5.2}).

Both  authors were  partially supported by ISF grant 687/13, BSF grant 2012247 and a Minerva foundation grant.
\section{Preliminaries}

In this section we review definitions and results on which we rely
in this work. For this, we mostly follow the definitions and notations
of \cite[Section 3]{BDOP13}, \cite[Section 4]{CGH14-2}, \cite{CGH16}
and \cite{Pas89} in Sections 2.2 and 2.3 and \cite[Sections 3.1-3.3]{AA16}
and \cite{AA18} in Section 2.4.

\subsection{The Denef-Pas language, definable sets and functions, motivic functions and motivic integration}

\label{subsec: Denef-Pas} In this section, we recall the definitions
of the Denef-Pas language, definable sets, definable functions and
motivic functions, and review an integration result concerning motivic
functions.

\subsubsection{The Denef-Pas language}

The Denef-Pas language, denoted $\Ldp$, is a first order language
with three sorts of variables: the valued field sort $\VF$, the residue
field sort $\RF$, and the value group sort $\VG$. The $\Ldp$ language
consists of the following: 
\begin{itemize}
\item The language of rings $\mathcal{L}_{\mathrm{Val}}=(+,-,\cdot,0,1)$
for the valued field sort $\VF$. 
\item The language of rings $\mathcal{L}_{\mathrm{Res}}=(+,-,\cdot,0,1)$
 for the residue field sort $\RF$. 
\item The language $\mathcal{L}_{\mathrm{Pres}}^{\infty}=\mathcal{L}_{\mathrm{Pres}}\cup\curly{\infty}$
for the value group sort $\VG$, where $\infty$ is a constant, and
$\mathcal{L}_{\mathrm{Pres}}=(+,-,\leq,\{\equiv_{\mathrm{mod}~n}\}_{n>0},0,1)$
is the Presburger language consisting of the language of ordered abelian
groups along with constants $0,1$ and a family of $2$-relations
$\{\equiv_{\mathrm{mod}~n}\}_{n>0}$ of congruence modulo $n$. 
\item A function $\val:\VF\rightarrow\VG$ for a valuation map. 
\item A function $\ac:\VF\rightarrow\RF$ for an angular component map. 
\end{itemize}
Altogether, we write 
\[
{\Ldp}=\{\mathcal{L}_{\mathrm{Val}},\mathcal{L}_{\mathrm{Res}},\mathcal{L}_{\mathrm{Pres}}^{\infty},\val,\ac\}.
\]

\subsubsection{The angular component map}

Let $K$ be a complete, discretely valued field, with valuation map $\val:K\to\Gamma\cup\{\infty\}$,
where $\Gamma$ is an ordered abelian group which we usually take
to be $\ints$. Let $\mathcal{O}_{K}$ be the ring of integers of
$K$ with maximal ideal $\mathfrak{m}_{K}=\{x\in K:\val(x)>0\}$,
and let $k_{K}=\mathcal{O}_{K}/\mathfrak{m}_{K}$ denote its residue
field and $\mathrm{Res}:\mathcal{O}_{K}\to k_{K}$ be the canonical
quotient map. An \textit{angular component} map $\mathrm{ac}:K\to k_{K}$
is a map satisfying the following: 
\begin{enumerate}
\item $\mathrm{ac}(0)=0$. 
\item $\mathrm{ac}|_{\mathcal{O}_{K}^{\times}}=\mathrm{Res}|_{\mathcal{O}_{K}^{\times}}$. 
\item $\mathrm{ac}|_{K^{\times}}$ is a multiplicative morphism from $K^{\times}$
to $k_{K}^{\times}$. 
\end{enumerate}
It is unique up to choosing a uniformizer $\pi\in\mathcal{O}_{K}$ and setting $\mathrm{ac}(\pi)=1$.

\subsubsection{Definable sets, definable functions and motivic functions}

 Any pair $(F,\pi)$ of a non-Archimedean local field $F$ and a uniformizer
$\pi$ of $\mathcal{O}_{F}$ has a natural ${\Ldp}$-structure. Let
$\mathrm{Loc}$ be the set of all such pairs, and $\mathrm{Loc}_{M}$
be the set of $(F,\pi)\in\mathrm{Loc}$ such that $F$ has residue
characteristic larger than $M$. Usually, we just write $F\in\mathrm{Loc}$,
omitting the second term, as our results are independent of the choice
of a uniformizer.

Given a formula $\phi$ in ${\Ldp}$, with $n_{1}$ free valued field
variables, $n_{2}$ free residue field variables and $n_{3}$ free
value group variables we can naturally interpret it in $F\in\mathrm{Loc}$,
yielding a subset $\phi(F)\subseteq F^{n_{1}}\times k_{F}^{n_{2}}\times\mathbb{Z}^{n_{3}}$.
We would like to study families of subsets of $F^{n_{1}}\times k_{F}^{n_{2}}\times\mathbb{Z}^{n_{3}}$
which arise from a fixed ${\Ldp}$-formula $\phi$. To do so, we introduce
the following definitions: 
\begin{defn}[{{{{See \cite[Definitions 2.3-2.6]{CGH16}}}}}]
\label{def:motivic function} Let $n_{1},n_{2},n_{3}$ and $M$ be
natural numbers. 
\begin{enumerate}
\item A collection $X=(X_{F})_{F\in\mathrm{Loc}_{M}}$ of subsets $X_{F}\subseteq F^{n_{1}}\times k_{F}^{n_{2}}\times\mathbb{Z}^{n_{3}}$
is called a \textit{definable set} if there is an ${\Ldp}$-formula
$\phi$ and $M'\in\nats$ such that $X_{F}=\phi(F)$ for every $F\in\mathrm{Loc}_{M'}$. 
\item We denote by $\mathrm{VF}$ and $\mathrm{RF}$ the definable sets
$(F)_{F\in\mathrm{Loc}}$ and $(k_{F})_{F\in\mathrm{Loc}}$ respectively. 
\item Let $X$ and $Y$ be definable sets. A \textit{definable function}
is a collection $f=(f_{F})_{F\in\mathrm{Loc}_{M}}$ of functions $f_{F}:X_{F}\rightarrow Y_{F}$,
such that the collection of their graphs $\{\Gamma_{f_{F}}\}_{F\in\mathrm{Loc}_{M}}$
is a definable set. 
\item Let $X$ be a definable set. A collection $h=(h_{F})_{F\in\mathrm{Loc}_{M}}$
of functions $h_{F}:X_{F}\rightarrow\mathbb{R}$ is called a \textit{motivic}
(or \textit{constructible}) function on $X$, if there exists $M'\in\nats$
such that for all $F\in\mathrm{Loc}_{M'}$ it can be written in the
following way (for every $x\in X_{F}$): 
\[
h_{F}(x)=\sum_{i=1}^{N}|Y_{i,F,x}|q_{F}^{\alpha_{i,F}(x)}\left(\prod_{j=1}^{N'}\beta_{ij,F}(x)\right)\left(\prod_{l=1}^{N''}\frac{1}{1-q_{F}^{a_{il}}}\right),
\]
where, 
\begin{itemize}
\item $N,N'$ and $N''$ are integers and $a_{il}$ are non-zero integers. 
\item $\alpha_{i}:X\rightarrow\mathbb{Z}$ and $\beta_{ij}:X\rightarrow\mathbb{Z}$
are definable functions.
\item $Y_{i,F,x}=\{y\in k_{F}^{r_{i}}:(x,y)\in{Y}_{iF}\}$ is the fiber over
$x$ where ${Y}_{i}\subseteq X\times\mathrm{RF}^{r_{i}}$ are definable
sets and $r_{i}\in\nats$. 
\item The integer $q_{F}$ is the size of the residue field $k_{F}$. 
\end{itemize}
The set of motivic functions on a definable set $X$ forms a ring,
which we denote by $\mathcal{C}(X)$. 
\end{enumerate}
\end{defn}

\subsubsection{Integration of motivic functions}

\label{subsect: Motivic integration} We want to show that the ring
of motivic functions is preserved under integration. In order to do
so, we first need to explain what does it mean to integrate a motivic
function. 

Take $\{\mu_{1,F}\}_{F\in\mathrm{Loc}}$ to be the family of normalized
Haar measures on $\VF$, that is $(\mu_{1,F})_{F}(\mathcal{O}_{F})=1$
for every $F\in\mathrm{Loc}$, and let $\{\mu_{2,F}\}_{F\in\mathrm{Loc}}$
and $\{\mu_{3,F}\}_{F\in\mathrm{Loc}}$ be the families of counting
measures on $\RF$ and $\ints$ respectively. 
Given a definable set $X\subseteq\VF^{n_{1}}\times\RF^{n_{2}}\times\ints^{n_{3}}$,
we can consider the family of measures $(\mu_{F})_{F\in\mathrm{Loc}}$
induced from the collection $(\mu_{1,F}^{n_{1}}\times\mu_{2,F}^{n_{2}}\times\mu_{3,F}^{n_{3}})_{F}$
for each $F\in\mathrm{Loc}$. We call measures such as $(\mu_F)_{F\in\mathrm{Loc}}$
\textit{motivic measures}. For other ways to obtain motivic measures,
see \cite[Section 8]{CL08}, \cite[Section 2.3]{CGH16} or \cite[Section 2.5]{CGH14-1}.
Using $(\mu_{F})_{F\in\mathrm{Loc}}$, we can integrate motivic functions.
 The following theorem shows that the ring of motivic functions is
preserved under integration with respect to any motivic measure: 
\begin{thm}[{{{{See \cite[Theorem 4.3.1]{CGH14-2} }}}}]
\label{integration of motivic} Let $X$ and $Y$ be ${\Ldp}$-definable
sets, $f$ be in $\mathcal{C}(X\times Y)$ and let $\mu$ be a motivic
measure on $Y$. Then there exist a function $g\in\mathcal{C}(X)$
and an integer $M>0$ such that for every $F\in\mathrm{Loc}_{M}$
and $x_{0}\in X_{F}$, if $f_{F}(x_{0},y)\in L^{1}(Y_{F})$ then 
\[
g_{F}(x)=\int_{Y_{F}}f_{F}(x,y)d\mu_{F}.
\]
\end{thm}

\begin{rem}
{In Section \ref{subsec:Motivic-measures-on} we extend the definition of motivic functions
and motivic measures to smooth algebraic $\rats$-varieties. }
\end{rem}

\subsection{Elimination of quantifiers and uniform cell decomposition}

In this subsection we state a quantifier elimination result and a
uniform cell decomposition theorem for the $\Ldp$-theory $\mathcal{T}_{\mathrm{H},\ac,0}$
which is defined below, and a Presburger cell decomposition theorem.

\subsubsection{Uniform cell decomposition in $\mathcal{T}_{\mathrm{H},\ac,0}$}

Denote by $\mathcal{T}_{\mathrm{H},\ac,0}$ the ${\Ldp}$-theory of
Henselian valued fields $K$ of residue characteristic zero such that
there is an angular component map $\mathrm{ac}:K\to k_{K}$. The theory
$\mathcal{T}_{\mathrm{H},\ac,0}$ has quantifier elimination (in the
valued field sort), and there is a uniform cell decomposition theorem.
This allows us, using specialization arguments, to obtain uniform
results about non-Archimedean local fields with residue characteristic
large enough: 
\begin{lem}[{{{See e.g. \cite[Lemma 3.5]{BDOP13}}}}]
\label{specialization arguments}Let $\phi$ be a sentence in ${\Ldp}$.
Assume that $\phi$ holds in all models of $\mathcal{T}_{\mathrm{H},\ac,0}$,
then there exists $M=M(\phi)\in\nats$ such that $\phi$ holds in
all $F\in\mathrm{Loc}_{M}$. 
\end{lem}

We now wish to state the uniform cell decomposition theorem. 
Let $f_{i}:\VF^{m}\times\VF\to\VF$ be functions who are polynomial
in their second variable, and whose coefficients are definable functions
in the first variable. The uniform cell decomposition theorem allows
us to decompose an $\Ldp$-definable set in $\VF^{m}\times\VF$ into
a disjoint union of smaller $\Ldp$-definable sets called cells, such
that on each cell $\val(f_{i})$ and $\ac(f_{i})$ have simpler description,
which depends on one less valued field variable.

For the sake of consistency and in order to avoid confusion, we follow
the definitions and notations of \cite[Section 3]{BDOP13} and \cite{Pas89}
with only minor changes. We start with the definition of a cell: 
\begin{defn}[{{{{Cells, see \cite[Definition 3.1]{BDOP13} or \cite[Definition 2.9]{Pas89}}}}}]
\label{def:Cell} ~ 
\begin{enumerate}
\item Let $y\in\VF$ and $x=(x_{1},\ldots,x_{m})\in\VF^{m}$ be valued field
sort variables, $\xi=(\xi_{1},\ldots,\xi_{n})\in\RF^{n}$ be residue
field sort variables and let $\lambda\in\ints_{>0}$. Furthermore
take a definable subset $C\subseteq\VF^{m}\times\RF^{n}$, and definable
functions $b_{1},b_{2},c:C\to\VF$. We also denote by $\square_{1},\square_{2}$
the relations $<,\leq$ or no condition. For each $\xi\in\RF^{n}$,
let $A(\xi)$ denote the definable set 
\[
\left\{ (x,y)\in\VF^{m}\times\VF:\,(x,\xi)\text{\ensuremath{\in}}C\wedge\val(b_{1}(x,\xi))\square_{1}\lambda\cdot\val(y-c(x,\xi))\square_{2}\val(b_{2}(x,\xi))\wedge\ac(y-c(x,\xi))=\xi_{1}\right\} .
\]
\item Suppose that the definable sets $A(\xi)$ are distinct; then $A=\bigcup\limits _{\xi\in\RF^{n}}A(\xi)$
is called a \textit{cell} in $\VF^{m}\times\VF$ with parameters $\xi$
and center $c$ and we call $A(\xi)$ a fiber of the cell $A$. Furthermore,
we denote by $\Theta(A)=(C,b_{1}(x,\xi),b_{2}(x,\xi),c(x,\xi),\lambda)$
the \textit{datum} of the cell $A$. 
\end{enumerate}
\end{defn}

\begin{defn}[{{{{Uniform cell decomposition, see \cite[Definition 3.2]{BDOP13}
or \cite[Theorem 3.2]{Pas89}}}}}]
\label{def:UCD} Let $y\in\VF$ and $x=(x_{1},\ldots,x_{m})\in\VF^{m}$,
and take $f_{1}(x,y),\ldots,f_{r}(x,y)$ to be polynomials in $y$
whose coefficients are definable functions in $x$, i.e. expressions
of the form $\sum\limits _{k}a_{k}(x)y^{k}$ where each $a_{k}$ is
a definable function. We say that a collection of valued fields $\mathcal{F}$
have a uniform cell decomposition of dimension $m$ with respect to
the functions $\curly{f_{i}}_{i=1}^{r}$ if, for some positive integer
$N$, there exist the following: 
\begin{itemize}
\item A non-negative integer $n$ and definable sets $C_{i}\subseteq\VF^{m}\times\RF^{n}$, 
\item Definable functions $b_{i,1}(x,\xi),b_{i,2}(x,\xi),c_{i}(x,\xi)$
and $h_{ij}(x,\xi)$ from $\VF^{m}\times\RF^{n}$ into $\VF$, 
\item Positive integers $\lambda_{i}$, 
\item Non-negative integers $w_{ij}$, 
\item Functions $\mu_{i}:\{1,\ldots,r\}\rightarrow\{1,\ldots,n\}$, where
$1\leq i\leq N$ and $1\leq j\leq r$, 
\end{itemize}
such that, 
\begin{itemize}
\item For any $F\in\mathcal{F}$, we have $F^{m}\times F=\bigcup\limits _{i=1}^{N}A_{i,F}$,
where the cells $A_{i}$ are defined by the cell data $\Theta(A_{i})=(C_{i},b_{i,1},b_{i,2},c_{i},\lambda_{i})$
for $1\leq i\leq N$. 
\item For every $1\leq i\leq N$ and $1\leq j\leq r$, and all $(x,y)\in A_{i}(\xi)$,
we have 
\[
\val(f_{j}(x,y))=\val(h_{ij}(x,\xi)(y-c(x,\xi))^{w_{ij}})\text{ and }\ac(f_{j}(x,y))=\xi_{\mu_{i}(j)},
\]
where the integers $w_{ij}$ and the maps $\mu_{i}$ do not depend
on $x,\xi$ and $y$. 
\end{itemize}
\end{defn}

The following uniform cell decomposition theorem was proved by Pas
for the ${\Ldp}$ language:
\begin{thm}[{{{{Uniform cell decomposition theorem, see \cite[Theorems 3.3, 3.6]{BDOP13}
or \cite[Theorems 3.1-3.2, Remark 3.3]{Pas89}}}}}]
\label{thm:(Uniform-cell-decomposition)} Let $y$ and $x=(x_{1},\ldots,x_{m})$
be valued field sort variables, and $f_{1}(x,y),\ldots,f_{r}(x,y)$
be polynomials in $y$ whose coefficients are definable functions
in $x$. Then, 
\begin{enumerate}
\item The class of all models of $\mathcal{T}_{\mathrm{H},\ac,0}$ has uniform
cell decomposition of dimension $m$ with respect to the functions
$\{f_{i}\}_{i=1}^{r}$. 
\item There is a constant $M=M(m;f_{1},\ldots,f_{r})$ such that $\mathrm{Loc}_{M}$
has uniform cell decomposition of dimension $m$ with respect to the
functions $\{f_{i}\}_{i=1}^{r}$. 
\end{enumerate}
\end{thm}

\subsubsection{Elimination of quantifiers in $\mathcal{T}_{\mathrm{H},\ac,0}$}

The following theorem is due to Denef and Pas: 
\begin{thm}[{{{{Quantifier elimination theorem, see \cite[Theorem 3.4]{BDOP13},
\cite[Theorem 2.1.1]{CL08} or \cite[Theorem 4.1]{Pas89}}}}}]
\label{Quantifier elimination theorem} The theory $\mathcal{T}_{\mathrm{H},\ac,0}$
admits elimination of quantifiers in the valued field sort. More explicitly,
every ${\Ldp}$-formula ${\psi(x,\xi,k) \subseteq \VF^{n_1} \times \RF^{n_2} \times \VG^{n_3}}$ 
is $\mathcal{T}_{\mathrm{H},\ac,0}$-equivalent to a finite
disjunction of formulas of the form 
\[
\chi(\ac(g_{1}(x)),\ldots,\ac(g_{s}(x)),\xi)\wedge\theta(\val(g_{1}(x)),\ldots,\val(g_{s}(x)),k)
\]
where $\chi$ is an $\mathcal{L}_{\mathrm{Res}}$ formula, $\theta$
is an $\mathcal{L}_{\mathrm{Pres}}$ formula and $g_{i}\in\mathbb{Z}[x_{1},\ldots,x_{n_1}]$. 
\end{thm}

\begin{notation}
\label{nota:Given-an--formula} Given an ${\Ldp}$-formula $\psi$,
we write $\Xi(\psi)=(\{g_{j}\}_{j=1}^{s},\{\chi_{i}\}_{i=1}^{r},\{\theta_{i}\}_{i=1}^{r})$
for its data, 
where $\psi$ is equivalent to a formula 
\[
\bigvee_{i=1}^{r}\chi_{i}(\ac(g_{1}(x)),\ldots,\ac(g_{s}(x)),\xi)\wedge\theta_{i}(\val(g_{1}(x)),\ldots,\val(g_{s}(x)),k),
\]
$x,\xi,k,\{g_{j}\}_{j=1}^s$ are as in Theorem \ref{Quantifier elimination theorem},
$\{\theta_{i}\}_{i=1}^r$ are $\mathcal{L}_{\mathrm{Pres}}^{\infty}$-formulas
and $\{\chi_{i}\}_{i=1}^r$ are $\mathcal{L}_{\mathrm{Res}}$-formulas. 
\end{notation}

\subsubsection{Presburger cell decomposition}

We conclude our discussion of cell decomposition results with a result
in the Presburger language. We begin with a definition. 
\begin{defn}[{{{{See \cite[Definition 1]{Clu03}}}}}]
\label{def:-Linear function} Let $X\subset\VG^m$ be an $\mathcal{L}_{\mathrm{Pres}}$-definable
set. 
\begin{enumerate}
\item We call a definable function $f:X\to\VG$ \textit{linear} if there
is a constant $\gamma\in\VG$ and integers $a_{i}$ and $0\leq c_{i}<n_{i}$
for $i=1,\ldots,m$ such that $x_{i}-c_{i}\equiv0\text{ }\mathrm{mod}\text{ }n_{i}$
and $f(x)=\sum\limits _{i=1}^{m}a_{i}(\frac{x_{i}-c_{i}}{n_{i}})+\gamma$ {for every $x\in X$}. 
\item We call a definable function $f:X\to\VG$ \textit{piecewise linear}
if there exists a finite partition of $X=\bigcup\limits _{i=1}^{N}A_{i}$,
such that $f|_{A_{i}}$ is linear for all $i$. 
\end{enumerate}
The following cell decomposition result will be used in the proof
of Theorem \ref{thm:L^1 constructible has polynomial decaying Fourier transform}: 
\end{defn}

\begin{thm}[{{{Presburger cell decomposition, see {\cite[Theorem 1]{Clu03}}}}}]
\label{Presburger Cell decomposition} Let $X\subseteq\VG^{m}$ and
$f:X\to\VG$ be $\mathcal{L}_{\mathrm{Pres}}$-definable. Then there
exists a finite partition $P$ of $X$ into cells (see \cite[Definition 2]{Clu03}),
such that the restriction $f|_{A}:A\to\VG$ is linear for each cell
$A\in P$. 
\end{thm}

\subsection{Resolution of singularities, rational singularities and the (FRS)
property}
{In this section we recall necessary notions from algebraic geometry, as well as define the (FRS) property
and state an analytic criterion for the (FRS) property.} 

\subsubsection{Resolution of singularities, rational singularities and definition
of the (FRS) property}

\begin{defn}
\label{def:ratlsings} Let $X$ be an algebraic variety over a field
$K$. 
\begin{enumerate}
\item A \textit{resolution of singularities} of $X$ is a proper map $p:\widetilde{X}\rightarrow X$
such that $\widetilde{X}$ is smooth and $p$ is a birational equivalence.
A \textit{strong resolution of singularities} of $X$ is a resolution
of singularities $p:\widetilde{X}\rightarrow X$ which is an isomorphism
over the smooth locus of $X$, denoted $X^{\mathrm{sm}}$. It is a
theorem of Hironaka \cite{Hir64}, that any $X$ over a field $K$
of characteristic zero admits a strong resolution of singularities. 
\item (See \cite[I.3 pages 50-51]{KKMS73} or \cite[Definition 6.1]{AA16})
We say that $X$ has \textit{rational singularities} if for any (or
equivalently, for some) resolution of singularities $p:\widetilde{X}\rightarrow X$,
the natural morphism $\mathcal{O}_{X}\rightarrow{R}p_{*}(\mathcal{O}_{\widetilde{X}})$
is a quasi-isomorphism, where ${R}p_{*}(-)$ is the higher direct
image functor. 
\end{enumerate}
\end{defn}

We now define the notion of an (FRS) map, which is the
relative version (i.e. for morphisms) of having rational singularities: 
\begin{defn}[{{{{\cite[Section 1.2.1, Definition II]{AA16}}}}}]
\label{def:FRS} Let $\varphi:X\to Y$ be a morphism between smooth
varieties $X$ and $Y$. 
\begin{enumerate}
\item We say that $\varphi:X\to Y$ is (FRS) at $x\in X(K)$ if it is flat
at $x$, and there exists an open $x\in U\subseteq X$ such that $U\times_{Y}\{\varphi(x)\}$
is reduced and has rational singularities. 
\item We say that $\varphi:X\to Y$ is (FRS) if it is flat and it is (FRS)
at $x$ for all $x\in X(\bar{K})$. 
\end{enumerate}
\end{defn}

A useful theorem is given in \cite{Elk78}, and implies in particular
that the (FRS)-locus of a morphism is open. 
\begin{thm}[{{{{See \cite[Theorem 6.3]{AA16} or \cite[Theorem 4,5]{Elk78}}}}}]
\label{Lemma Elkik} Let $\varphi:X\rightarrow S$ be a flat morphism
of finite type $K$-schemes and let $x\in X$ be such that $\varphi(x)$
is a rational singularity in $S$. Assume that $x$ is a rational
singularity of its fiber $\varphi^{-1}(\varphi(x))$, then we have
the following: 
\begin{enumerate}
\item $x$ is a rational singularity in $X$. 
\item The set $\left\{ x\in X(K):x\text{ is a rational singularity of }\varphi^{-1}(\varphi(x))\right\} $
is open in $X(K)$. 
\end{enumerate}
\end{thm}

\subsubsection{Measures on $p$-adic analytic varieties and an analytic criterion
for the (FRS) property}

Let $X$ be a $d$-dimensional smooth algebraic variety over $K$.
We denote by $\Omega_{X}^{r}$ the sheaf of differential $r$-forms
on $X$ and by $\Omega_{X}^{r}[X]$ (resp. $\Omega_{X}^{r}(X)$) the
regular $r$-forms (resp. rational $r$-forms). Given a non-Archimedean
local field $F\supset K$, then $X(F)$ has a structure of an $F$-analytic
manifold.

For $\omega\in\Omega_{X}^{\mathrm{top}}(X)$, we can define a measure
$\left|\omega\right|_{F}$ on $X(F)$ as follows. Let $U\subseteq X(F)$
be a compact open set and let $\phi$ be an $F$-analytic {diffeomorphism
from an open subset $W\subseteq F^{d}$ to $U$}. We can write
$\phi^{*}\omega=gdx_{1}\wedge{\ldots}\wedge dx_{d}$, for some $g:W\rightarrow F$,
and define 
\[
\left|\omega\right|_{F}(U)=\int_{W}\left|g\right|_{F}d\lambda,
\]
where $|\cdot|_{F}$ is the normalized absolute value on $F$ and
$\lambda$ is the normalized Haar measure on $F^{d}$ (i.e. $\lambda(\mathcal{O}_{F}^{d})=1$).
Note that this definition is independent of the diffeomorphism $\phi$,
and that the measure $|\omega|_{F}$ obtained in this way is unique
after fixing $\omega$.

\begin{defn}
\label{Schwartz measure} Let $X$ be as above. 
\begin{enumerate}
\item A measure $\mu$ on $X(F)$ is called \textit{smooth} if every point
$x\in X(F)$ has an analytic neighborhood $U$ and an ($F$-analytic)
diffeomorphism $\phi:U\rightarrow\mathcal{O}_{F}^{d}$ such that $\phi_{*}\mu|_{U}$
is a Haar measure on $\mathcal{O}_{F}^{d}$. 
\item A measure on $X(F)$ is called \textit{Schwartz} if it is smooth and
compactly supported. 
\item A measure $\mu$ on $X(F)$ has \textit{continuous density}, if there
is a smooth measure $\widetilde{\mu}$ and a continuous function $f:X(F)\rightarrow\mathbb{C}$
such that $\mu=f\cdot\widetilde{\mu}$. 
\end{enumerate}
\end{defn}

Schwartz measures and measures with continuous density can be characterized
in the following way: 
\begin{prop}[{\cite[Proposition 3.3]{AA16}}]
\label{Prop 2.13} Let $X$ be a smooth
variety over a non-Archimedean local field $F$. 
\begin{enumerate}
\item A measure $\mu$ on $X(F)$ is Schwartz if and only if it is a linear
combination of measures of the form $f\left|\omega\right|_{F}$, where
$f$ is a locally constant and compactly supported function on $X(F)$,
and $\omega\in\Omega_{X}^{\mathrm{top}}(X)$ has no zeroes or poles
in the support of $f$. 
\item A measure $\mu$ on $X(F)$ has continuous density if and only if
for every point $x\in X(F)$ there is a neighborhood $U$ of $x$,
a continuous function $f:U\rightarrow\mathbb{C}$, and $\omega\in\Omega_{X}^{\mathrm{top}}(X)$
with no poles in $U$ such that $\mu=f\left|\omega\right|_{F}$ {on $U$}. 
\end{enumerate}
\end{prop}

We can now state an analytic criterion which is equivalent to the
(FRS) property. It is often much easier to use this criterion (specifically
the third condition) than to use Definition \ref{def:FRS} directly. 
\begin{thm}[{\cite[Theorem 3.4]{AA16}}]
\label{Analytic condition for (FRS)} Let
$\varphi:X\rightarrow Y$ be a map between smooth algebraic varieties
defined over a finitely generated field $K$ of characteristic $0$,
and let $x\in X(K)$. Then the following conditions are equivalent: 
\begin{enumerate}
\item $\varphi$ is (FRS) at $x$. 
\item There exists a Zariski open neighborhood $x\in U\subseteq X$ such
that, for any non-Archimedean local field $F\supseteq K$ and any
Schwartz measure $\mu$ on $U(F)$, the measure $(\varphi|_{U(F)})_{*}(\mu)$
has continuous density. 
\item For any finite extension $K'/K$, there exists a non-Archimedean local
field $F\supseteq K'$ and a non-negative Schwartz measure $\mu$
on $X(F)$ that does not vanish at $x$ such that $(\varphi|_{X(F)})_{*}(\mu)$
has continuous density. 
\end{enumerate}
\end{thm}
\section{Reduction of Theorem \ref{Main result} to an analytic problem}

\subsection{Convolution of morphisms preserves smoothness properties}

We would like to show that the convolution operation preserves certain
properties of morphisms, and in particular that it preserves the (FRS)
property (see Definition \ref{def:(FRS)}). We use the following proposition: 
\begin{prop}
\label{prop:convolution preserves} Let $X$ and $Y$ be varieties
over a field $K$, let $G$ an algebraic $K$-group and let $S$ be
a property of morphisms that is preserved under base change and compositions.
If $\psi:Y\to G$ is arbitrary, $\varphi:X\to G$ is a morphism that
satisfies property $S$, and the natural map $i_{K}:Y\to\mathrm{Spec}(K)$
has property $S$, then $\varphi*\psi$ has property $S$. 
\end{prop}

\begin{proof}
Since $i_{K}$ satisfies $S$ and $S$ is preserved under base change,
the projection to the first coordinate $\pi_{G}:G\times Y\to G$ satisfies
$S$. Now consider the following fibered diagram: 
\[
\begin{array}{ccc}
X\times Y & \overset{\pi_{X}}{\longrightarrow} & X\\
\downarrow\alpha & \, & \downarrow\varphi\\
G\times Y & \overset{\beta}{\longrightarrow} & G
\end{array},
\]
where $\alpha(x,y)=(\varphi\ast\psi(x,y),y)$ and $\beta(g,y)=g\cdot\psi(y)^{-1}$.
This implies that $\alpha$ satisfies $S$, and since $\varphi\ast\psi=\pi_{G}\circ\alpha$
we are done. 
\end{proof}
\begin{prop}[The (FRS) property is preserved under compositions and base change]
\label{FRS is a good property} Let $X,Y$ and $Z$ be smooth $K$-varieties,
and let $\varphi:X\rightarrow Y$ be an (FRS) morphism. 
\begin{enumerate}
\item If $\psi:Y\to Z$ is (FRS), then $\psi\circ\varphi$ is (FRS). 
\item Consider the following base change diagram, 
\[
\begin{array}{ccc}
X\times_{Y}Z & \overset{\widetilde{\varphi}}{\longrightarrow} & Z\\
\downarrow & \, & \downarrow\psi\\
X & \overset{\varphi}{\longrightarrow} & Y
\end{array}
\]
where $\psi:Z\to Y$ is arbitrary. Then $\widetilde{\varphi}$ is
(FRS). 
\end{enumerate}
\end{prop}
\begin{proof}~
\begin{enumerate} 
\item Since flatness is preserved by compositions, we have that $\psi\circ\varphi$
is flat. As a consequence, for any $z\in Z$ the fiber $X_{z}:=(\psi\circ\varphi)^{-1}(z)$
is a local complete intersection scheme, and in particular Cohen-Macauley.
By the $(S_{1}+R_{0})$-criterion (see e.g. \cite[Lemma 10.151.3]{stacks-project}),
in order to show that $X_{z}$ is reduced it is enough to show that
$X_{z}$ is generically reduced, or equivalently, 
that its non-smooth locus is of codimension $\geq1$. By \cite[III.10.2]{Har77},
since $\psi\circ\varphi$ is flat, the smooth locus of $X_{z}$ is
equal to the set $X_{z}^{\mathrm{sm},\psi\circ\varphi}:=\{x\in X_{z}:\psi\circ\varphi\text{ is smooth at }x\}$.
Thus, we would like to show that $X_{z}^{\mathrm{sm},\psi\circ\varphi}$
is dense in $X_{z}$. As above, define $X_{y},Y_{z}$ and $X_{y}^{\mathrm{sm},\varphi},Y_{z}^{\mathrm{sm},\psi}$
for any $y\in Y$ and $z\in Z$. Since $\psi$ and $\varphi$ are
(FRS), we have that $Y_{z}^{\mathrm{sm},\psi}$ is dense in $Y_{z}$
and $X_{y}^{\mathrm{sm},\varphi}$ is dense in $X_{y}$ for any $y\in Y$.
Since smoothness of morphisms is preserved under composition, we have
$X_{z}^{\mathrm{sm},\psi\circ\varphi}\supseteq\bigcup_{y\in Y_{z}^{\mathrm{sm},\psi}}X_{y}^{\mathrm{sm},\varphi}$.
Now let $U$ be an open set in $X_{z}$. By flatness, $\varphi^{-1}(Y_{z}^{\mathrm{sm},\psi})$
is open and dense in $X_{z}$, thus there exists $y\in Y_{z}^{\mathrm{sm},\psi}$
such that $U\cap X_{y}$ is a non-empty open subset of $X_{y}$, and
thus $U\cap X_{y}^{\mathrm{sm},\varphi}$ is non empty. This shows
that $X_{z}^{\mathrm{sm},\psi\circ\varphi}$ is dense in $X_{z}$,
and hence $X_{z}$ is reduced. We therefore showed that $\psi\circ\varphi$
is flat, with reduced fibers.

Now for $z\in Z$, let $x\in X_{z}$ and consider the map $\varphi|_{X_{z}}:X_{z}\rightarrow Y_{z}$.
By our assumption, $y:=\varphi(x)$ is a rational singularity of $Y_{z}$
and $x$ is a rational singularity of $X_{y}$. Since $\varphi$ is
flat and $\varphi|_{X_{z}}$ is a base change of $\varphi$, it follows
that $\varphi|_{X_{z}}$ is flat as well. By Theorem \ref{Lemma Elkik},
$x$ is a rational singularity of $X_{z}$. Hence, the fibers of $\psi\circ\varphi$
have rational singularities and we are done. 
\item First, notice that the fibers of $\widetilde{\varphi}$ are the base
change of the fibers of $\varphi$. Indeed, for every $y\in Y$ and
$z\in Z$ such that $\psi(z)=y$, we have: 
\begin{align*}
\spec(K(\curly{z}))\times_{Z}\left(Z\times_{Y}X\right) & \simeq\spec(K(\curly{z}))\times_{Y}X\\
 & \simeq\spec(K(\curly{z}))\times_{\mathrm{Spec}(K(\curly{y}))}\left(\mathrm{Spec}(K(\curly{y}))\times_{Y}X\right).
\end{align*}
Since reduceness, having rational singularities {and flatness are preserved under
base change (recall that $\mathrm{char}(K)=0$), we deduce that the
fibers of $\widetilde{\varphi}$ are reduced and have rational singularities and that $\widetilde{\varphi}$ is flat.
Therefore that $\widetilde{\varphi}$ is (FRS).}
\end{enumerate}
\end{proof}

By Propositions \ref{prop:convolution preserves} and \ref{FRS is a good property}
above, it is immediate that the convolution operation preserves the
(FRS) property. The same holds for dominance, flatness and smoothness. 
\begin{cor}
\label{cor:convolution presrves (FRS)} Let $G$ be an algebraic $K$-group,
and suppose that $\varphi:X\rightarrow G$ is (FRS) (resp. dominant/flat/smooth)
and let $\psi:Y\rightarrow G$ be any morphism. Then the morphisms
$\varphi\ast\psi:X\times Y\rightarrow G$ and $\psi*\varphi:X\times Y\rightarrow G$
are (FRS) (resp. dominant/flat/smooth). 
\end{cor}

\subsection{Reduction of Theorem \ref{Main result} to the case of strongly dominant
morphisms}

We want to show that under reasonable assumptions, high enough convolution
power of a given morphism yields a morphism whose restriction to every
absolutely irreducible component is dominant. This will imply it is
enough to prove Theorem \ref{Main result} for such morphisms. 
\begin{defn}
\label{def:strongly generating} Let $G$ be an algebraic group,
let $\varphi:X\rightarrow G$ be a morphism of $K$-varieties and
let $\curly{X_{i}}_{i=1}^{l}$ be the absolutely irreducible components
of $X$. 
\begin{enumerate}
\item We say that $\varphi$ is \textit{generating} if $\varphi(X)\not\subseteq gH$
for any algebraic subgroup $H\leq G$ and $g\in G(\overline{K})$. 
\item We say that $\varphi$ is \textit{strongly generating} if $\varphi_{|X_{i}}$
is generating for all $1\leq i\leq l$. 
\item We say that $\varphi$ is \textit{strongly dominant} if $\varphi_{|X_{i}}$
is dominant 
for all $1\leq i\leq l$. 
\end{enumerate}
\end{defn}

\begin{prop}
\label{prop:morphismdom} Let $X$ be a smooth $K$-variety, $G$
be a commutative algebraic $K$-group and $\varphi:X\to G$ be a strongly
generating morphism. Then there exists $n\in\nats$ such that $\varphi^{n}$
is strongly dominant. 
\end{prop}

\begin{proof}
Assume that $X$ is absolutely irreducible. By exchanging $\varphi$
with its translation by $g\in G(\overline{K})$, we may assume that
$e\in\varphi(X)$. Set $U_{n}:=\mathrm{Im}(\varphi^{n})$ and note
that $U_{n}\subseteq U_{n+1}$ for all $n\in\mathbb{N}$. By dimension
considerations, there exists $n_{0}\in\mathbb{N}$ such that $\overline{U_{n}}=\overline{U_{m}}$
for all $m,n>n_{0}$, and in particular $\overline{U_{n}^{2}}=\overline{U_{2n}}=\overline{U_{n}}\subseteq\overline{U_{n}}^{2}$.

Now, since the multiplication map $m:G\times G\to G$ is an open map,
by Chevalley's theorem, $U_{n}^{2}$ contains an open set in $\overline{U_{n}}^{2}$,
and thus by the irreducibility of $\overline{U_{n}}^{2}=m(\overline{\mathrm{Im}\varphi^{n}}\times\overline{\mathrm{Im}\varphi^{n}})$
we get that $\overline{U_{n}^{2}}=\overline{U_{n}}^{2}$. Setting
$H:=\overline{U_{n}}$ for $n$ large enough, we get that $H\cdot H=H$,
so $H$ is a closed algebraic semigroup.

Given $h\in H(\overline{K})$, since $L_{h}:H\to H$ (left translation
by $h$) is an injective map, by the Ax-Grothendieck theorem (\cite{Ax68,Gro66})
we deduce that it is also surjective (over $\overline{K}$). As $e\in\mathrm{Im}L_{h}$
for all $h\in H(\overline{K})$, we get that $H$ is an algebraic
group. Our assumption implies that $H=G$, and hence $\varphi^{n}$
is dominant.

We now move to prove the general case. Let $\curly{X_{i}}_{i=1}^{l}$
be the absolutely irreducible components of $X$. By the above argument,
since $\varphi$ is strongly generating there exist $n_{i}\in\nats$
such that $\varphi_{|X_{i}\times\ldots\times X_{i}}^{n_{i}}$ is dominant
for all $i$. Set $n=\max\limits _{i}\curly{n_{i}}$, we claim that
$\varphi^{nl}$ is strongly dominant. 
Indeed, all the absolutely irreducible components of $X^{nl}$ are
of the form $X_{i_{1}}\times\ldots\times X_{i_{nl}}$, where $1\leq i_{k}\leq l$,
and therefore there exists some $1\leq j\leq l$ such that $X_{j}$
appears at least $n$ times in $X_{i_{1}}\times\ldots\times X_{i_{nl}}$.
Since dominance is preserved by convolution, and $G$ is commutative,
we are done. 
\end{proof}
As a corollary, we get the desired reduction: 
\begin{cor}
\label{Cor-reduction to dominant} It is enough to prove Theorem \ref{Main result}
for $\varphi:X\rightarrow V$ {strongly} dominant. 
\end{cor}

\subsection{Motivic measures on algebraic $\rats$-varieties}

\label{subsec:Motivic-measures-on} The goal of this subsection is
to define the notion of a motivic measure on a smooth algebraic
$\mathbb{Q}$-variety $X$. This will allow us to construct a collection of measures
$\{\mu_{p}\}_{p\text{ prime}}$ on $\{X(\Zp)\}_{p}$ which behave
well with respect to pushforward under a strongly dominant map $\varphi:X\rightarrow V$.
For such a collection $\{\mu_{p}\}_{p\text{ prime}}$, we will be
able to show (in Section \ref{sec:Fourier-transform-of}) that after
sufficiently many self convolutions of $\varphi_{*}(\mu_{p})$, we
get a measure with continuous density, and that the number of convolutions
required does not depend on $p$. Using Proposition \ref{Reduction to many Qp},
we will then deduce our main theorem (Theorem \ref{Main result})
for the case $K=\rats$.

Let $X$ be a reduced, finite type affine $\ints$-scheme. An embedding
$\psi:X\hookrightarrow\mathbb{A}_{\ints}^{N}$ naturally gives rise
to an ${\Ldp}$-definable subset $\{\psi(X)(F)\}_{F\in\mathrm{Loc}}$
of $\VF^{N}$. This allows us to define definable subsets and motivic
functions on $\rats$-varieties. 
\begin{defn}
~\label{def:motivic functions on var} 
\begin{enumerate}
\item Let $X$ be a finite type, affine $\ints$-scheme. 
\begin{enumerate}
\item A collection $\{Y_{F}\}_{F\in\mathrm{Loc}}$ of subsets $Y_{F}\subseteq X(F)$
is called a \textit{definable subset} of $X$ if there exists an embedding
$\psi:X\hookrightarrow\mathbb{A}_{\ints}^{N}$ such that $\{\psi(Y_{F})\}_{F\in\mathrm{Loc}}$
is a definable subset of $\curly{\psi(X)(F)}_{F\in\mathrm{Loc}}$.
\item A collection $h=(h_{F})_{F}$ of functions $h_{F}:X(F)\rightarrow\mathbb{R}$
is called a \textit{motivic function} on $X$, and denoted $h\in\mathcal{C}(X)$,
if there exists an embedding $\psi:X\hookrightarrow\mathbb{A}_{\ints}^{N}$
and $f\in\mathcal{C}(\psi(X))$ such that $h=\psi^{*}(f)$. 
\end{enumerate}
\item Let $X$ be a finite type $\ints$-scheme. 
\begin{enumerate}
\item A collection $\{Y_{F}\}_{F\in\mathrm{Loc}}$ of subsets $Y_{F}\subseteq X(F)$
is called a \textit{definable subset} of $X$ if there exists an affine
cover $X=\bigcup\limits _{i=1}^{l}U_{i}$, with embeddings $\psi_{i}:U_{i}\hookrightarrow\mathbb{A}_{\ints}^{N_{i}}$,
such that $\{\psi_{i}(U_{i}(F)\cap Y_{F})\}_{F\in\mathrm{Loc}}$ is
a definable subset of $\psi_{i}(U_{i})$ for all $1\leq i\leq l$.
\item A collection $h=(h_{F})_{F}$ of functions $h_{F}:X(F)\rightarrow\mathbb{R}$
is called a \textit{motivic function} on $X$ if there exists an affine
cover $X=\bigcup\limits _{i=1}^{l}U_{i}$, with embeddings $\psi_{i}:U_{i}\hookrightarrow\mathbb{A}_{\ints}^{N_{i}}$,
and a collection $f_{1},\ldots,f_{l}$ where $f_{i}\in\mathcal{C}(\psi_{i}(U_{i}))$
and $\psi_{i}^{*}(f_{i})=h|_{U_{i}}$. 
\end{enumerate}
\item Let $X$ be an algebraic $\rats$-variety. 
\begin{enumerate}
\item A collection $\{Y_{F}\}_{F\in\mathrm{Loc}}$ of subsets $Y_{F}\subseteq X(F)$
is called a \textit{definable subset} of $X$ if there exists a $\ints$-model\footnote{Recall that a $\ints$-model of $X$ is a $\ints$-scheme $\widetilde{X}$
such that $\widetilde{X}\times_{\spec(\ints)}\spec(\rats)\simeq X$.} $\widetilde{X}$ of $X$, such that $\{Y_{F}\}_{F\in\mathrm{Loc}}$
is a definable subset of $\widetilde{X}$. We denote the set of definable
subsets of $X$ by $\mathcal{D}(X)$. 
\item A collection $h=(h_{F})_{F}$ of functions $h_{F}:X(F)\rightarrow\mathbb{R}$
is called a \textit{motivic function} on $X$, if there exists a $\ints$-model
$\widetilde{X}$ of $X$ such that $h\in\mathcal{C}(\widetilde{X})$. 
\end{enumerate}
\end{enumerate}
\end{defn}

\begin{rem}
Note that the notions above are independent of the embedding $\psi$
into affine space. Given two embeddings $\psi:X\hookrightarrow\mathbb{A}_{\ints}^{N_{1}}$
and $\psi':X\hookrightarrow\mathbb{A}_{\ints}^{N_{2}}$, we have an
algebraic $\ints$-isomorphism between $\psi(X)$ and $\psi'(X)$,
which induces a definable isomorphism $\{\psi(X)_{F}\}_{F\in\mathrm{Loc}}\xrightarrow{\sim}\{\psi'(X)_{F}\}_{F\in\mathrm{Loc}}$. 
\end{rem}

\begin{lem}
Let $X$ be an algebraic $\rats$-variety, let $Y=\{Y_{F}\}_{F\in\mathrm{Loc}}$
be a collection of subsets $Y_{F}\subseteq X(F)$ and let $h=(h_{F})_{F}$
be a collection of functions $h_{F}:X(F)\rightarrow\mathbb{R}$. 
\begin{enumerate}
\item $\{Y_{F}\}_{F\in\mathrm{Loc}}$ is a definable subset of $X$ if and
only if for any $\ints$-model $\widetilde{X}$ of $X$, we have that
$\{Y_{F}\}_{F\in\mathrm{Loc}}$ is a definable subset of $\widetilde{X}$. 
\item $h=(h_{F})_{F}$ is a motivic function on $X$ if and only if for
any $\ints$-model $\widetilde{X}$ of $X$, we have that $h\in\mathcal{C}(\widetilde{X})$. 
\end{enumerate}
\end{lem}

\begin{proof}
We prove (1), the proof for (2) is similar. For any two $\ints$-models
$\widetilde{X}_{1}$ and $\widetilde{X}_{2}$ there exists a $\rats$-isomorphism
$\varphi:\widetilde{X}_{1}\times_{\spec(\ints)}\spec(\rats)\rightarrow\widetilde{X}_{2}\times_{\spec(\ints)}\spec(\rats)$.
Let $Y=\{Y_{F}\}_{F\in\mathrm{Loc}}$ be a definable subset of $\widetilde{X}_{2}$
and consider the collection\textcolor{red}{{} }$\{\varphi^{-1}(Y_{F})\}_{F\in\mathrm{Loc}}$,
where $\varphi^{-1}(Y_{F})\subseteq X(F)$. We want to show that $\{\varphi^{-1}(Y_{F})\}_{F\in\mathrm{Loc}}$
is a definable subset of $\widetilde{X}_{1}$. It is enough to prove
for the case where $\widetilde{X}_{1}$ and $\widetilde{X}_{2}$ are
affine with $\widetilde{X}_{1}\hookrightarrow\mathbb{A}_{\ints}^{N_{1}}$
and $\widetilde{X}_{2}\hookrightarrow\mathbb{A}_{\ints}^{N_{2}}$.
By Theorem \ref{Quantifier elimination theorem}, $Y$ is defined
by an ${\Ldp}$-formula $\widetilde{\phi}$ which is a disjunction
of formulas of the form 
\[
\chi_{i}(\ac(g_{1}(x)),\ldots,\ac(g_{s}(x)))\wedge\theta_{i}(\val(g_{1}(x)),\ldots,\val(g_{s}(x))),
\]
where $\chi_{i}$ is an $\mathcal{L}_{\mathrm{Res}}$-formula, $\theta_{i}$
is an $\mathcal{L}_{\mathrm{Pres}}$-formula and $g_{j}\in\mathbb{Z}[x_{1},\ldots,x_{N_{2}}]$.
Notice that if the formula $\phi$ with data $\Xi({\phi})=(\{g_{j}\}_{j=1}^{s},\{\chi_{i}\}_{i=1}^{r},\{\theta_{i}\}_{i=1}^{r})$
(recall Notation \ref{nota:Given-an--formula}) defines a set $Z\subseteq\widetilde{X}_{2}(F)$,
then the formula $\varphi^{*}\phi$ with data $\Xi(\varphi^{*}{\phi})=(\{g_{j}\circ\varphi\}_{j=1}^{s},\{\chi_{i}\}_{i=1}^{r},\{\theta_{i}\}_{i=1}^{r})$
defines the set $\varphi^{-1}(Z)\subseteq\widetilde{X}_{1}(F)$. Thus,
our main candidate for a formula for $\{\varphi^{-1}(Y_{F})\}_{F\in\mathrm{Loc}}$
is $\varphi^{*}\phi$. 

The problem which arises is that $\curly{g_{j}\circ\varphi}_{j=1}^{s}\subset\rats[y_{1},\ldots,y_{N_{1}}]$
do not necessarily have integral coefficients. In order to solve this,
define $\Xi(\phi'):=(\{N\cdot g_{j}\circ\varphi\}_{j=1}^{s},\{\chi_{i}\}_{i=1}^{r},\{\theta_{i}\}_{i=1}^{r})$
for $N\in\nats$ large enough such that $N\cdot g_{j}\circ\varphi\in\ints[y_{1},\ldots,y_{N_{1}}]$
for any $j$. Notice that for $M\in\nats$ large enough, we have $\ac(N\cdot g_{j}\circ\varphi(x))=N\cdot\ac(g_{j}\circ\varphi(x))$
and $\val(N\cdot g_{j}\circ\varphi(x))=\val(g_{j}\circ\varphi(x))$
for any $F\in\mathrm{Loc}_{M}$, so we only need to take care of $\curly{\chi_{i}}_{i=1}^{r}$.

It is left to show that for any $\mathcal{L}_{\mathrm{Res}}$-formula
$\chi(t_{1},\ldots,t_{s})$ there exists an $\mathcal{L}_{\mathrm{Res}}$-formula
$\chi'$ such that $\chi'(t_{1},\ldots,t_{s})=\chi(N\cdot t_{1},\ldots,N\cdot t_{s})$,
and then we are done, by setting $\Xi(\phi''):=(\{N\cdot g_{j}\circ\varphi\}_{j=1}^{s},\{\chi'_{i}\}_{i=1}^{r},\{\theta_{i}\}_{i=1}^{r})$,
and observing that $\phi''$ defines $\{\varphi^{-1}(Y_{F})\}_{F\in\mathrm{Loc}}$.

Firstly, {every $\mathcal{L}_{\mathrm{Res}}$-formula} $\chi(t_{1},\ldots,t_{s})$ is defined by zeros
of polynomials $P_{j}(t_{1},\ldots,t_{s})$, and possibly has quantifiers.
Let $D_{j}$ be the maximal degree in each $P_{j}$ and consider the
polynomials $\widetilde{P}_{j}$ which are obtained by replacing each
monomial $M_{ij}$ in $P_{j}$ by $N^{D_{j}-\deg(M_{ji})}M_{ij}$.\textcolor{red}{{}
}
Now, notice that $\widetilde{P}_{j}(N\cdot t_{1},\ldots,N\cdot t_{s})=N^{D_{j}}\cdot P_{j}(t_{1},{\ldots},t_{s})$,
and define $\chi'$ by replacing each $P_{j}$ by $\widetilde{P}_{j}$
in $\chi$. It is easy to see that $\chi'$ satisfies $\chi'(t_{1},{\ldots},t_{s})=\chi(N\cdot t_{1},{\ldots},N\cdot t_{s})$
and we are done. 
\end{proof}
As a conclusion, we can pull back definable sets and motivic functions
with respect to $\rats$-morphisms. 
\begin{lem}
\label{Pulback of motivic} Let $X$ and $Y$ be two algebraic
$\rats$-varieties and let $\varphi:X\rightarrow Y$ be a $\rats$-morphism.
Then for any definable subset $\{Z_{F}\}_{F\in\mathrm{Loc}}$ of $Y$
we have that $\{\varphi^{-1}(Z_{F})\}_{F\in\mathrm{Loc}}$ is a definable
subset of $X$ and for any $f\in\mathcal{C}(Y)$ we have $f\circ\varphi\in\mathcal{C}(X)$. 
\end{lem}

\begin{proof}
We may assume that $\varphi:X\rightarrow Y$ is defined over $S^{-1}\ints$,
the localization of $\ints$ by a finite set of primes $S$. Since
$S^{-1}\ints$ is of finite type over $\ints$ we may choose a $\ints$-model
$\widetilde{\varphi}:\widetilde{X}\rightarrow\widetilde{Y}$ of the
morphism $\varphi$ (which includes $\ints$-models of $X$ and $Y$).
Since in the setting of the $\Ldp$ language pullbacks are well defined,
by reducing to the affine case, we have well defined pullbacks $\widetilde{\varphi}^{*}:\mathcal{C}(\widetilde{Y})\rightarrow\mathcal{C}(\widetilde{X})$
and $\widetilde{\varphi}^{*}:\mathcal{D}(\widetilde{Y})\rightarrow\mathcal{D}(\widetilde{X})$.

Since $\widetilde{X}\times_{\spec(\ints)}\spec(\rats)\simeq X$ and
$\widetilde{Y}\times_{\spec(\ints)}\spec(\rats)\simeq Y$, and since
$\rats\subset F$ for any $F\in\mathrm{Loc}$, there are identifications
$\widetilde{X}(F)\simeq X(F)$ and $\widetilde{Y}(F)\simeq Y(F)$,
under which $\varphi$ and $\widetilde{\varphi}$ induce the same
map $X(F)\rightarrow Y(F)$. This implies the lemma. 
\end{proof}
The next lemma follows easily by reducing to the affine case and choosing
a $\ints$-model. 
\begin{lem}
\label{lem: operations on definable sets} Let $X$ be an algebraic
$\rats$-variety. 
\begin{enumerate}
\item Any $\rats$-subvariety $Y\subseteq X$ is definable. 
\item $\mathcal{D}(X)$ is closed under intersections, unions and complements. 
\end{enumerate}
\end{lem}

\begin{defn}
\label{definition motivic measure} Let $X$ be a smooth algebraic
$\mathbb{Q}$-variety. We say that a collection of measures $\mu=\{\mu_{F}\}_{F\in\mathrm{Loc}}$
on $\{X(F)\}_{F\in\mathrm{Loc}}$ is a \textit{motivic} measure on
$X$ if there exists an open affine cover $X=\bigcup\limits _{j=1}^{l}U_{j}$,
such that $\mu_{F}|_{U_{j}(F)}=(f_{j})_{F}\left|\omega_{j}\right|_{F}$,
where $f_{j}\in\mathcal{C}(U_{j})$ and $\omega_{j}$ is a non-vanishing
top form on $U_{j}$. 
\end{defn}

\begin{lem}
\label{lem: equivalent def to motivic measure}Let $X$ be a smooth
algebraic $\mathbb{Q}$-variety and let $\mu=\{\mu_{F}\}_{F\in\mathrm{Loc}}$
be a collection of measures on $\curly{X_{F}}_{F\in\mathrm{Loc}}$.
Then $\mu$ is motivic if and only if there exists an open affine cover $X=\bigcup\limits _{i=j}^{l}U_{j}$,
such that $\mu$ can be written as $\mu_{F}:=\sum\limits _{j=1}^{l}f_{j,F}\cdot\left|\omega_{j}\right|_{F}$
for some $f_{j}\in\mathcal{C}(U_{j})$ and $\omega_{j}$ non-vanishing
top forms on $U_{j}$ respectively. 
\end{lem}

\begin{proof}
Let $\mu_{F}:=\sum\limits _{j=1}^{l}f_{j,F}\cdot\left|\omega_{j}\right|_{F}$.
For each $j$, notice that ${\omega_{j}}_{|_{U_{i}\cap U_{j}}}=h_{ji}\cdot{\omega_{i}}_{|_{U_{i}\cap U_{j}}}$
where $h_{ji}$ is a non-vanishing regular function on $U_{i}\cap U_{j}$.
Hence we have
\[
\mu_{F}|_{U_{i}(F)}
=\sum\limits _{j=1}^{l}f_{j,F}\cdot\left|{\omega_{j}}_{|_{U_{i}\cap U_{j}}}\right|_{F}
=\sum\limits _{j=1}^{l} 1_{U_i(F) \cap U_j(F)}f_{j,F}\cdot\left|h_{ji}\right|_{F}\left|\omega_{i}\right|_{F}.
\]
Since $\sum f_{j,F}\cdot\left|h_{ji}\right|_{F}$ is a motivic function,
we see that $\mu$ is motivic. The other direction is similar. 
\end{proof}
\begin{prop}
\label{prop:existsmotivicschwartz} Let $X$ be a smooth algebraic
$\mathbb{Q}$-variety, then there exists a motivic measure $\mu=\{\mu_{F}\}_{F\in\mathrm{Loc}}$
on $X$, such that for every $F\in\mathrm{Loc}$, $\mu_{F}$ is a
non-negative Schwartz measure and ${\supp}(\mu_{F})=X(\mathcal{O}_{F})$. 
\end{prop}

\begin{proof}
Choose an affine open cover $X=\bigcup\limits _{i=1}^{l}U_{i}$, with
embeddings $\psi_{i}:U_{i}\hookrightarrow\mathbb{A}_{\rats}^{N_{i}}$
and non-vanishing top forms $\omega_{i}$ on $U_{i}$ (it is possible
since $X$ is smooth). We can construct a disjoint open cover $\{V_{i}\}_{i=1}^{l}$
of $X(\mathcal{O}_{F})$ by setting $V_{1}=U_{1}(\mathcal{O}_{F})$
and $V_{i}=U_{i}(\mathcal{O}_{F})\backslash\bigcup\limits _{j=1}^{i-1}U_{j}(\mathcal{O}_{F})$.

Define the measure $\mu_{F}:=\sum\limits_{i=1}^l 1_{V_{i}}\cdot\left|\omega_{i}\right|_{F}$.
Then
by Proposition \ref{Prop 2.13}, $\mu_{F}$ is a Schwartz measure
and it is clearly non-negative and supported on $X(\mathcal{O}_{F})$
for any $F\in\mathrm{Loc}$. It is left to show that $\{\mu_{F}\}_{F\in\mathrm{Loc}}$
is motivic. By Lemmas \ref{lem: operations on definable sets} and
\ref{lem: equivalent def to motivic measure}, it is enough to show
that $U_{i}(\mathcal{O}_{F})$ is a definable subset of $U_{i}$ for
each $i$.

Choose $\ints$-models $\widetilde{U}_{i}$ of $U_{i}$ and $\widetilde{\psi}_{i}:\widetilde{U}_{i}\hookrightarrow\mathbb{A}_{\ints}^{N_{i}}$
of $\psi_{i}$, and let 
\[
B_{1}^{N_{i}}(0)=\{\vec{y}\in\VF^{N_{i}}:\mathrm{val}(y_{j})\geq0\text{ for }0\leq j\leq N_{i}\}
\]
be the ${\Ldp}$-definable set whose points over $F\in\mathrm{Loc}$
are the unit disc in $F^{N_{i}}$. Clearly, the set \[\left\{ \widetilde{\psi}_{i}(\widetilde{U}_{i})(F)\cap B_{1}^{N_{i}}(0)(F)\right\} _{F\in\mathrm{Loc}}\]
is ${\Ldp}$-definable for any $i$. This proves the proposition.
\end{proof}

\subsection{Reduction to an analytic problem}

Our next goal is to use the equivalent characterization of the (FRS)
property given in Theorem \ref{Analytic condition for (FRS)} in order
to reduce our main problem, as stated in Theorem \ref{Main result}
(with $K=\rats$), to an analytic question (see Proposition \ref{Reduction to many Qp}).
To do so, we need the following lemma: 
\begin{lem}
\label{lem:finitenumberfieldextension} Let $X$ be a smooth algebraic
$\rats$-variety and let $x_{1},\ldots,x_{k}\in X(\overline{\mathbb{Q}})$.
Then for any finite extension $K/\mathbb{Q}$ such that $x_{1},\ldots,x_{k}\in X(K)$
there exist infinitely many prime numbers $p$ with $i_{p}:K\hookrightarrow\Qp$
such that $i_{p}{}_{*}(x_{1}),\ldots,i_{p}{}_{*}(x_{k})\in X(\mathbb{Z}_{p})$. 
\end{lem}

\begin{proof}
Let $K$ be such that $x_{1},\ldots,x_{k}\in X(K)$ and let $K'$
be the Galois closure of $K$. We can write $K'=\rats(\alpha)$ where
$\alpha$ is a root of a monic minimal polynomial $x^{r}+\sum\limits _{i=0}^{r-1}a_{i}x^{i}=q(x)\in\ints[x]$.
Denote $\overline{q}(x)$ for the reduction modulo $p$ of $q(x)$,
then by a conclusion of Chebotarev's density theorem, $\overline{q}(x)$
splits in $\mathbb{F}_{p}$ for infinitely many primes $p$. 
Set 
\[
S_{N}=\{p\text{ prime}:\overline{q}(x)\text{ splits in }\mathbb{F}_{p}[x]\text{ and }p>N\}
\]
where $N=r\max\limits _{i}|a_{i}|$, then $\overline{q}(x)$ is separable
in $\mathbb{F}_{p}[x]$ for every $p\in S_{N}$.

Now, use Hensel's lemma for each $p\in S_{N}$ to lift a root of $\bar{q}(x)\in\mathbb{F}_{p}[x]$
to a root $\alpha'$ of $q(x)$ in $\ints_{p}$. This gives rise to
an embedding $i_{p}:K'\hookrightarrow\Qp$ by $\alpha\mapsto\alpha'$,
where $\alpha$ maps to $\Zp$ as $\alpha'$ lies in $\Zp$, which
in turn gives rise to a map $i_{p_{*}}:X(K')\to X(\Qp)$. 

Finally, for each $x_{i}\in X(K)\subset X(K')$ take an open affine
neighborhood $U_{i}$ of $x_{i}$ and let $\psi:U_{i}\hookrightarrow\mathbb{A}_{K'}^{N_{i}}$
be a closed embedding, for some $N_{i}\in\nats$. We may write each
$x_{i}\in U_{i}(K')\subseteq(K')^{N_{i}}$ in the form $x_{i}=(x_{i,1}(\alpha'),\ldots,x_{i,N_{i}}(\alpha'))$
where $x_{i,j}(\alpha')=\sum\limits _{t=0}^{r-1}\frac{b_{i,j,t}}{c_{i,j,t}}(\alpha')^{t}$
with $b_{i,j,t},c_{i,j,t}\in\ints$. Taking $N'=\max\{\{|c_{i,j,t}|\}_{i,j,t},N\}$,
we get that $i_{p_{*}}(x_{1}),\ldots,i_{p_{*}}(x_{k})\in X(\Zp)$
for all $p\in S_{N'}$. 
\end{proof}
Using Theorem \ref{Analytic condition for (FRS)} and Lemma \ref{lem:finitenumberfieldextension}
we can now reduce our main problem to the following. 
\begin{prop}
\label{Reduction to many Qp} Let $X$ be a smooth algebraic
$\rats$-variety, $\mu$ be a motivic measure on $X$ as in Proposition \ref{prop:existsmotivicschwartz},
and let $\varphi:X\to\mathbb{A}_{\mathbb{Q}}^{m}$ be a strongly dominant
morphism. Assume that there exists $n\in\mathbb{N}$, such that for
every large enough prime $p$ the measure $\varphi_{*}^{n}(\mu_{\mathbb{Q}_{p}}\times\ldots\times\mu_{\mathbb{Q}_{p}})$
has continuous density with respect to the normalized Haar measure
on $(\Qp)^{m}$. Then the map $\varphi^{n}:X\times\ldots\times X\to\mathbb{A}_{\mathbb{Q}}^{m}$
is (FRS). 
\end{prop}

\begin{proof}
Let $x=(x_{1},\ldots,x_{n})\in\left(X\times\ldots\times X\right)(\overline{\mathbb{Q}})$.
There exists a finite extension $K/\mathbb{Q}$ such that $x_{1},\ldots,x_{n}\in X(K)$.
By \cite[Theorem 3.4]{AA16}, in order to show that $\varphi^{n}$
is (FRS) at $x$ it is enough to show that for any finite extension
$K'/K$, there exists a non-Archimedean local field $F$ containing
$K'$ and a non-negative Schwartz measure $\mu$ on $\left(X\times\ldots\times X\right)(F)$
that does not vanish at $x$, such that $\varphi_{*}^{n}(\mu)$ has
continuous density. Given such $K'/K$, by Lemma \ref{lem:finitenumberfieldextension},
we can choose $F=\mathbb{Q}_{p}$ for large enough $p$ such that
$K'\hookrightarrow\mathbb{Q}_{p}$ and $i_{p}{}_{*}(x_{1}),\ldots,i_{p}{}_{*}(x_{n})\in X(\mathbb{Z}_{p})$.
By our assumption, $\varphi_{*}^{n}(\mu_{\mathbb{Q}_{p}}\times\ldots\times\mu_{\mathbb{Q}_{p}})$
has continuous density with respect to the normalized Haar measure
on $(\mathbb{Q}_{p})^{m}$, and $\mu_{\mathbb{Q}_{p}}\times\ldots\times\mu_{\mathbb{Q}_{p}}$
does not vanish at $x$, so we are done. 
\end{proof}

\section{Pushforward of a motivic measure under a strongly dominant map}

In the last section we have reduced Theorem \ref{Main result}, in
the case $K=\rats$, to an analytic question on the pushforward of
motivic Schwartz measures under a strongly dominant morphism $\varphi:X\to\mathbb{A}_{\mathbb{Q}}^{m}$
(Proposition \ref{Reduction to many Qp}), where $X$ is a smooth
$\rats$-variety. In this section, we show that the pushforward
under $\varphi$ of a motivic Schwartz measure $\mu=\{\mu_{F}\}_{F\in\mathrm{Loc}}$
yields a motivic measure $\{\varphi_{*}(\mu_{F})\}_{F\in\mathrm{Loc}}$
with $L^{1}$-density with respect to the normalized Haar measure
on $F^{m}$ (see Corollary \ref{main result on push of morivic}).
This will be a conclusion of the more general Theorem \ref{thm:pushofconstructiblefunctionunderdominant},
whose proof relies on the following consequence of \cite[Theorem 10.1.1]{CL08}: 
\begin{lem}
\label{lem:sumonfibersisconstructible} Suppose that $X$ and $Y$
are smooth algebraic $\rats$-varieties, $\varphi:X\to Y$ is a morphism
with finite fibers and let $f\in\mathcal{C}(X)$. Then the function
$I_{f,F}(y)=\sum_{x\in\varphi^{-1}(y)}f_{F}(x)$ is in $\mathcal{C}(Y)$. 
\end{lem}

\begin{proof}
Let $\bigcup\limits _{j}V_{j}$ be an open affine cover of $Y$, and
for each $j$ take an open affine cover $\bigcup\limits _{i=1}^{r}U_{ji}$
of $\varphi^{-1}(V_{j})$. By the definition 
of a motivic function on a $\rats$-variety, it is enough to prove
the lemma for $\varphi_{j}:=\varphi|_{\varphi^{-1}(V_{j})}:\varphi^{-1}(V_{j})\rightarrow V_{j}$.
By choosing a $\ints$-model for the diagram consisting of $X,Y,\{V_{j}\}_{j=1}^s$,
$\{U_{ji}\}_{i=1}^r$, their intersections, and the morphisms between these
objects, we can assume they are all defined over $\ints$. Construct
a definable disjoint cover of $\varphi^{-1}(V_{j})$ by setting $U'_{j1}=U_{j1}$,
and $U'_{ji}=U_{ji}\backslash\bigcup\limits _{k=1}^{i-1}U_{jk}$,
and let $\varphi_{ji}:=\varphi|_{U'_{ji}}$ and $f_{ji}:=f|_{U'_{ji}}$.
Then we have the following: 
\[
I_{f_{j},F}(y)=\sum\limits _{i=1}^{r}\sum_{x\in\varphi_{ji}^{-1}(y)}f_{ji,F}(x)=\sum\limits _{i=1}^{r}\varphi_{ji*}(f_{ji}).
\]
Since by \cite[Theorem 10.1.1]{CL08} every summand of the RHS of
the above formula is a constructible function on $Y$, we are done. 
\end{proof}
Using Lemma \ref{lem:sumonfibersisconstructible}, we can now show
the following variant of \cite[Proposition 5.6 and Corollary 3.6]{AA16}: 
\begin{thm}
\label{thm:pushofconstructiblefunctionunderdominant} Let $X$ and
$Y$ be smooth algebraic $\rats$-varieties, let $M\in\nats$, and
let $\varphi:X\to Y$ be a strongly dominant morphism. Let $\mu_{X}$
be a motivic measure on $X$ such that $\mu_{X,F}$ is a Schwartz
measure for every $F\in\mathrm{Loc}_{M}$, and let $\mu_{Y}$ be a
motivic measure on $Y$ such that $\mu_{Y,F}$ is smooth and non-vanishing
on $Y(F)$ for every $F\in\mathrm{Loc}_{M}$. 
\begin{enumerate}
\item $\varphi_{*}(\mu_{X})$ is a motivic measure, and $\varphi_{*}(\mu_{X,F})$
is absolutely continuous with respect to $\mu_{Y,F}$ for any $F\in\mathrm{Loc}_{M}$. 
\item In particular, if $Y$ has a non-vanishing top form $\omega_{Y}$,
then there exists $g\in\mathcal{C}(Y)$ such that $\varphi_{*}(\mu_{X,F})$
is absolutely continuous with respect to $\left|\omega_{Y}\right|_{F}$
with density $g_{F}$, for any $F\in\mathrm{Loc}_{M}$. 
\end{enumerate}
\end{thm}

\begin{proof}
Since $Y$ is smooth, by Definition \ref{definition motivic measure}
we may assume that $\Omega_{Y/\mathbb{Q}}^{\mathrm{top}}$ is free,
and in particular that $Y$ has a non-vanishing top form $\omega_{Y}$.
Hence, it is enough prove part 2) of the theorem. By Lemma \ref{lem: equivalent def to motivic measure}
and the smoothness of $X$, we may assume that $\Omega_{X/\mathbb{Q}}^{\mathrm{top}}$
is free and that there exists $f\in \mathcal{C}(X)$ such that $\mu_{X,F}=f_{F}\left|\omega_{X}\right|_{F}$
for some top form $\omega_{X}$.

Denote by $X^{\mathrm{sm},\varphi}$ the smooth locus of $\varphi$
and by $i_{\mathrm{sm}}:X^{\mathrm{sm},\varphi}\hookrightarrow X$
the inclusion. Since $\varphi$ is strongly dominant, we have by \cite[Corollary 3.6]{AA16}
that for any $F\in\mathrm{Loc}_{M}$, the measure $\varphi_{*}(f_{F}\left|\omega_{X}\right|_{F})$
is absolutely continuous with respect to $\left|\omega_{Y}\right|_{F}$
and has an $L^{1}$-density $g_{F}$ such that 
\[
g_{F}(y)=\int_{\varphi^{-1}(y)\cap i_{\mathrm{sm}}(X^{\mathrm{sm},\varphi})(F)}f_{F}(x)\left|\frac{\omega_{X}}{\varphi^{*}\omega_{Y}}|_{\varphi^{-1}(y)\cap i_{\mathrm{sm}}(X^{\mathrm{sm},\varphi})}\right|_{F}.
\]

To finish the theorem, we need to show that $g\in\mathcal{C}(Y)$.
Let $j:U\hookrightarrow X^{\mathrm{sm},\varphi}$ be an open dense
affine subvariety of $X^{\mathrm{sm},\varphi}$ and set $\psi:=i_{\mathrm{sm}}\circ j:U\hookrightarrow X$
and $\varphi_{U}:=\varphi\circ\psi:U\to Y$. Since $\varphi_{U}$
is a smooth map and $\Omega_{X/\mathbb{Q}}^{\mathrm{top}}$ and $\Omega_{Y/\mathbb{Q}}^{\mathrm{top}}$
are free, $\varphi_{U}$ factors as: 
\[
\varphi_{U}:U\overset{\widetilde{\varphi}}{\to}\mathbb{A}_{\rats}^{\mathrm{dim}X-\mathrm{dim}Y}\times Y\overset{\pi}{\to}Y,
\]
where $\pi$ is the projection, and $\widetilde{\varphi}$ is an \'{e}tale
map. Since $U$ is open and dense in $X^{\mathrm{sm},\varphi}$, we
have that $\psi(U(F))$ is dense in $X^{\mathrm{sm},\varphi}(F)$
for any $F\in\mathrm{Loc}$. 
This implies 
\begin{align*}
\int_{\varphi^{-1}(y)\cap i_{\mathrm{sm}}(X^{\mathrm{sm},\varphi})(F)}f_{F}\cdot\left|\frac{\omega_{X}}{\varphi^{*}\omega_{Y}}|_{\varphi^{-1}(y)\cap i_{\mathrm{sm}}(X^{\mathrm{sm},\varphi})}\right|_{F} & =\int_{\varphi^{-1}(y)\cap\psi(U(F))}f_{F}\cdot\left|\frac{\omega_{X}}{\varphi^{*}\omega_{Y}}\right|_{F}\\
 & =\int_{\varphi_{U}^{-1}(y)(F)}(f\circ\psi)_{F}\left|\frac{\psi^{*}(\omega_{X})}{(\varphi_{U})^{*}\omega_{Y}}\right|_{F},
\end{align*}
and therefore 
\[
\varphi_{*}(f_{F}\left|\omega_{X}\right|_{F})=\varphi_{U}{}_{*}\left((f\circ\psi)_{F}\left|\psi^{*}(\omega_{X})\right|_{F}\right)=\pi_{*}\left(\widetilde{\varphi}_{*}\left((f\circ\psi)_{F}\left|\psi^{*}(\omega_{X})\right|_{F}\right)\right).
\]
Now, by Lemma \ref{Pulback of motivic}, we have that $f\circ\psi\in\mathcal{C}(U)$.
As $\widetilde{\varphi}$ is \'{e}tale, it has finite fibers, so we can
write: 
\[
\widetilde{\varphi}_{*}\left((f\circ\psi)_{F}\left|\psi^{*}(\omega_{X})\right|_{F}\right)=h_{F}\left|\omega\times\omega_{Y}\right|_{F},
\]
where $\omega$ is a top form on $\mathbb{A}_{\rats}^{\mathrm{dim}X-\mathrm{dim}Y}$
which induces the normalized Haar measure $\lambda=\left|\omega\right|_{F}$
on $F^{\mathrm{dim}X-\mathrm{dim}Y}$, and 
\[
h_{F}(t,y)=\sum_{x\in\widetilde{\varphi}^{-1}(t,y)(F)}(f\circ\psi)_{F}(x)\left|\frac{\psi^{*}(\omega_{X})}{\widetilde{\varphi}^{*}\left(\omega\times\omega_{Y}\right)}\right|_{F}(x).
\]
By Lemma \ref{lem:sumonfibersisconstructible}, we have that $h=\{h_{F}\}_{F\in\mathrm{Loc}}$
is in $\mathcal{C}(\mathbb{A}_{\rats}^{\mathrm{dim}X-\mathrm{dim}Y}\times Y)$.
Notice that for any $F\in\mathrm{Loc}_{M}$, we have $\left(g_{F}\left|\omega_{Y}\right|_{F}\right)(y)=\pi_{*}\left(h_{F}\left|\omega\times\omega_{Y}\right|_{F}\right)(y)$,
that is: 
\[
g_{F}(y)=\int_{F^{\mathrm{dim}X-\mathrm{dim}Y}\times\{y\}}h_{F}(t,y)d\lambda.
\]
Using \cite[Theorem 4.3.1]{CGH14-2}, we have $g=\{g_{F}\}_{F\in\mathrm{Loc}_{M}}\in\mathcal{C}(Y)$
and we are done. 
\end{proof}
As a conclusion, we obtain the following result, which is the main
goal of this section: 
\begin{cor}
\label{main result on push of morivic} Let $X$ be a smooth algebraic
$\rats$-variety, $\varphi:X\to\mathbb{A}_{\rats}^{m}$ be a strongly
dominant morphism and $\mu$ a motivic measure on $X$ as in Proposition
\ref{prop:existsmotivicschwartz}. Then there exist $M\in\mathbb{N}$
and $f\in\mathcal{C}(\mathbb{A}_{\rats}^{m})$, such that for every
$p>M$ the measure $\varphi_{*}(\mu_{\Qp})$ is absolutely continuous
with respect to the normalized Haar measure on $\mathbb{Q}_{p}^{m}$
with density $f_{\mathbb{Q}_{p}}\in{L}^{1}$. 
\end{cor}

\section{\label{sec:Fourier-transform-of}Uniform bounds on decay rates of
Fourier transform of motivic measures}

In this section we finish the proof of Theorem \ref{Main result}
for the case $k=\rats$: 
\begin{thm}
\label{Main result-Theorem 5.1}Let $X$ be a smooth $\rats$-variety
and let $V=\mathbb{A}_{\rats}^{m}$ be the $m$-dimensional affine
space. Let $\varphi:X\to V$ be a strongly generating morphism (Definition
\ref{def:strongly generating}), then there exists $N\in\mathbb{N}$
such that for any $n>N$, the $n$-th convolution power $\varphi^{n}$
is (FRS). 
\end{thm}

Let $\mu$ be a motivic measure on $X$ as in Proposition \ref{prop:existsmotivicschwartz}.
We saw in Corollary \ref{Cor-reduction to dominant} that $\varphi$
can be taken to be a strongly dominant morphism. We have further showed,
using Proposition \ref{Reduction to many Qp}, that it is enough to
show that there exists $n\in\mathbb{N}$, such that for large enough
prime $p$, the measure $\varphi_{*}^{n}(\mu_{\mathbb{Q}_{p}}\times\ldots\times\mu_{\mathbb{Q}_{p}})$
has continuous density with respect to the normalized Haar measure
on $(\mathbb{Q}_{p})^{m}$.

Notice that 
\[
\varphi_{*}^{n}(\mu_{\mathbb{Q}_{p}}\times\ldots\times\mu_{\mathbb{Q}_{p}})=\varphi_{*}(\mu_{\mathbb{Q}_{p}})*\dots*\varphi_{*}(\mu_{\mathbb{Q}_{p}}),
\]
and that by Corollary \ref{main result on push of morivic}, the measure
$\{\varphi_{*}(\mu_{F})\}_{F}$ is motivic, and for any $F\in\mathrm{Loc}$
we have that $\varphi_{*}(\mu_{F})$ is compactly supported and has
${L}^{1}$-density. Hence, our next goal is to show that given a motivic
measure $\sigma=\{\sigma_{F}\}_{F\in\mathrm{Loc}}$ on $\VF^{m}$
such that $\sigma_{F}$ is compactly supported and has ${L}^{1}$-density
for all $F\in\mathrm{Loc}$, then there exists $N\in\nats$, such
that the $N$-th convolution power $\{\sigma_{F}^{N}\}_{F\in\mathrm{Loc}}$
has continuous density for any $F\in\mathrm{Loc}_{M}$ and $M$ large
enough.

Recall that the Fourier transform $\mathcal{F}(f)$ of an ${L}^{1}$-function
$f:\Qp^{n}\rightarrow\complex$ is a continuous function, and that
$\mathcal{F}\circ\mathcal{F}(f)(x)=f(-x)$. Thus, in order to show
that $\sigma_{F}*\dots*\sigma_{F}$ is continuous, it is enough to
show that $\mathcal{F}(\sigma_{F}*\dots*\sigma_{F})=\mathcal{F}(\sigma_{F})^{N}$
is in ${L}^{1}$ for some $N$ that does not depend on $F\in\mathrm{Loc}$,
but firstly, we need to make sense of the Fourier transform of a motivic
function (or measure).

In \cite[Section 7]{CL10}, Cluckers and Loeser defined a class of motivic exponential functions
and an analogue of Fourier transform for this class of functions (see also
 \cite[Section 4]{CGH14-2} and \cite[Section 2]{CGH16}).
Specializing to a non-Archimedean local field $F$, the Fourier transform
operation can be established as follows:

Fix a collection $\{\psi_{F}\}_{F\in\mathrm{Loc}}$ of non-trivial
additive characters by $\psi_{F}(x)=\mathrm{exp}{}^{\frac{2\pi i}{p}\cdot\mathrm{Tr}_{k_{F}/\mathbb{F}_{p}}(\overline{x})}$
for any $x\in\mathcal{O}_{F}$, where $k_{F}=\mathcal{O}_{F}/\mathfrak{m}_{F}$
is the residue field (of characteristic $p$), $\overline{x}$ is
the reduction modulo $\mathfrak{m}_{F}$ of $x$ and $\mathrm{Tr}_{k_{F}/\mathbb{F}_{p}}$
is the trace. Let $\langle\,,\,\rangle:F^{m}\times F^{m}\to F$ be
the standard inner product on $F^{m}$ through which we identify $F^{m}$
with $\left(F^{m}\right)^{\vee}$. Notice that $\psi_{F}$ is trivial
on $\mathfrak{m}_{F}$. We denote by $\lambda_{F}$ the normalized
Haar measure on $F^{m}$ (i.e $\lambda(\mathcal{O}_{F}^{m})=1$).
For $f\in\mathcal{C}(\VF^{m})$ with $f_{F}$ absolutely integrable for any $F\in\mathrm{Loc}$, we define the Fourier transform by
\[
\mathcal{F}(f_{F})(y)=\int_{F^{m}}f_{F}(x)\psi_{F}(\langle y,x\rangle)dx.
\]

We wish to prove the following uniform variant of \cite[Theorem 4.1]{Clu04}
for the ${\Ldp}$-language: 
\begin{thm}
\label{thm:L^1 constructible has polynomial decaying Fourier transform}
Let $h\in\mathcal{C}(\VF^{m})$. Assume that there exist a definable
set $L\subseteq\VF^{m}$ and a natural number $M'$ such that for
any $F\in Loc_{M'}$ the set $L_{F}$ is compact, $\supp(h_{F})\subset L_{F}$
and that $\left|h_{F}\right|$ is integrable. Then there exist a real
constant $\alpha<0$ and a natural number $M>M'$, such that for any
$F\in\mathrm{Loc}_{M}$ 
\[
\left|\mathcal{F}(h_{F})(y)\right|<d(F)\cdot\mathrm{min}\{\left|y\right|{}^{\alpha},1\},
\]
for some constant $d(F)$ which depends on $F$. 
\end{thm}

\begin{cor}
Theorem \ref{thm:L^1 constructible has polynomial decaying Fourier transform}
implies Theorem \ref{Main result-Theorem 5.1}. 
\end{cor}

\begin{proof}
Let $\mu$ be a motivic measure on $X$ as in Proposition \ref{prop:existsmotivicschwartz}.
By Theorem \ref{thm:L^1 constructible has polynomial decaying Fourier transform},
there exists a real constant $\alpha<0$ such that for any $F\in\mathrm{Loc}_{M}$
\[
\left|\mathcal{F}(\varphi_{*}(\mu_{F}))(y)\right|<d(F)\cdot\mathrm{min}\{\left|y\right|{}^{\alpha},1\}.
\]
Set $N=\left\lceil -\frac{2}{\alpha}\right\rceil $, then we have
\[
\left|\mathcal{F}(\varphi_{*}^{N}(\mu_{F}\times\dots\times\mu_{F}))(y)\right|
=\left|\mathcal{F}(\varphi_{*}(\mu_{F})(y))\right|^N
<d(F)^{N}\cdot\mathrm{min}\{\left|y\right|^{-2},1\}
\]
and in particular it is ${L}^{1}$. Thus $\varphi_{*}^{N}(\mu_{F}\times\dots\times\mu_{F})$
has continuous density for any $F\in\mathrm{Loc}_{M}$, for some $M\in\nats$,
and by Proposition \ref{Reduction to many Qp} this implies Theorem
\ref{Main result-Theorem 5.1}. 
\end{proof}
\subsection{Proof of Theorem \ref{thm:L^1 constructible has polynomial decaying Fourier transform}} 
We now prove Theorem \ref{thm:L^1 constructible has polynomial decaying Fourier transform}.
Firstly, it is clear that $\mathcal{F}(h_{F})$ is bounded since $\left|\mathcal{F}(h_{F})\right|\leq\int_{F^{m}}\left|h_{F}(x)\right|dx<\infty$.
It is enough to show, for $1\leq i\leq m$, that $\left|\mathcal{F}(h_{F})(y)\right|<C_{i}(F)\cdot\left|y_{i}\right|^{\alpha_{i}}$
for some $\alpha_{i}<0$ and $C_{i}(F)>0$. We prove this by reducing
to a one dimensional analogue of the problem (see Lemma \ref{lemma 5.4})
which is easier to solve (Proposition \ref{One dimensional lemma}). 
\begin{lem}
\label{lemma 5.4} Let $M'\in\nats$, and let the motivic function
$h\in\mathcal{C}(\VF^{m})$ and the definable set $L\subseteq\VF^{m}$
be as in Theorem \ref{thm:L^1 constructible has polynomial decaying Fourier transform}.
Then there exist $M\in\nats$ and a motivic function $g\in\mathcal{C}(\VF)$
such that for any $F\in\mathrm{Loc}_{M}$ and $y_{1},\ldots,y_{m}\in F$
we have 
\[
|\mathcal{F}(h_{F})(y_{1},\ldots,y_{m})|\leq g_{F}(y_{m})\text{ and }\lim\limits _{\left|y_{m}\right|\to\infty}g_{F}(y_{m})=0.
\]
\end{lem}

\begin{proof}
By Theorem \ref{Quantifier elimination theorem}, the function $h\in\mathcal{C}(\VF^{m})$
is determined by finitely many polynomials $\{g_{j}\}_{j=1}^{s}$
in $\mathbb{Z}[x_{1},{\ldots},x_{m}]$. Write $x=(\hat{x},x_{m})$
with $\hat{x}=(x_{1},{\ldots},x_{m-1})$. By uniform cell decomposition
applied to the functions $\{g_{j}\}_{j=1}^{s}$ (see Theorem \ref{thm:(Uniform-cell-decomposition)}),
there exist cells $A_{i}$ with cell data $\Theta(A_{i})=(C_{i},b_{i,1},b_{i,2},c_{i},\lambda_{i})$
such that $F^{m}=\bigcup\limits _{i=1}^{N}A_{i,F}$ and $A_{i,F}=\bigcup\limits _{\xi\in k_{F}^{r}}A_{i,F}(\xi)$,
where each fiber $A_{i}(\xi)$ is as in Definition \ref{def:Cell}.
Furthermore, for any $1\leq i\leq N$ and $1\leq j\leq s$, and any $(\hat{x},x_{m})\in A_{i}(\xi)$,
we have 
\[
\val(g_{j,F}(\hat{x},x_{m}))=\val(h_{ij,F}(\hat{x},\xi)(x_{m}-c_{i,F}(\hat{x},\xi))^{w_{ij}})\text{ and }\ac(g_{j,F}(\hat{x},x_{m}))=\xi_{\mu_{i}(j)},
\]
where $\mu_{i},w_{ij},h_{ij}$ are as in Definition \ref{def:UCD}.
Hence, for $y=(\hat{y},y_{m})\in F^{m-1}\times F$ we have, 
\[
\mathcal{F}(h_{F})(y)=\int_{F^{m}}h_{F}(\hat{x},x_{m})\psi_{F}(\langle y,x\rangle)dx_{m}d\hat{x}=\sum_{i=1}^{N}\int_{A_{i,F}}h_{F}(\hat{x},x_{m})\psi_{F}(\langle y,x\rangle)dx_{m}d\hat{x}.
\]
Let $A\in\{A_{1},\ldots,A_{N}\}$ be one of the above cells. For simplicity,
we assume it has datum $\Theta(A)=(C,b_{1},b_{2},c,\lambda)$, and
similarly we replace $\mu_{i},w_{ij}$ and $h_{ij}$ with $\mu,w_{j}$
and $h_{j}$. Note that 
\begin{align*}
A_{F} & =\bigcup_{\xi\in k_{F}^{r}}A_{F}(\xi)=\bigcup_{\xi\in k_{F}^{r}}\bigcup_{l\in\ints}\left\{ (\hat{x},x_{m})\in A_{F}(\xi):\val(x_{m}-c(\hat{x},\xi))=l\right\} \\
 & =\bigcup_{\xi\in k_{F}^{r}}\bigcup_{l\in\ints}\{(\hat{x},x_{m})\in F^{m}:\hat{x}\in W_{F}(l,\xi)\wedge x_{m}\in B_{F}(l,\xi,c,\hat{x})\},
\end{align*}
where 
\[
B(l,\xi,c,\hat{x}):=\{x_{m}\in\VF:\val(x_{m}-c(\hat{x},\xi))=l,\,\ac(x_{m}-c(\hat{x},\xi))=\xi_{1}\},
\]
\[
W(l,\xi):=\{\hat{x}\in\VF^{m-1}:(\hat{x},\xi)\text{\ensuremath{\in}}C,\val(b_{1}(\hat{x},\xi))\square_{1}\lambda\cdot l\square_{2}\val(b_{2}(\hat{x},\xi))\},
\]
and $\square_{1,2}$ are either $\leq,<$ or no condition, as in the
cell $A_{F}$ (recall Definition \ref{def:Cell}). Calculating the
integral over $A_{F}$ yields: 
\begin{flalign}
\int_{A_{F}}h_{F}(\hat{x},x_{m})\psi_{F}(\langle y,x\rangle)dx_{m}d\hat{x}
=&\sum_{\xi\in k_{F}^{r}}\sum_{l\in\mathbb{Z}}\int_{W_{F}(l,\xi)}\int_{B_{F}(l,\xi,c,\hat{x})}h_{F}(\hat{x},x_{m})\psi_{F}(\langle\hat{y},\hat{x}\rangle+y_{m}x_{m})dx_{m}d\hat{x} \notag \\
=&\sum_{\xi\in k_{F}^{r}}\sum_{l\in\mathbb{Z}}\int_{W_{F}(l,\xi)}\psi_{F}(\langle\hat{y},\hat{x}\rangle+c_{F}(\hat{x},\xi)\cdot y_{m}) \notag \\
&\cdot\left(\int_{B_{F}(l,\xi,c,\hat{x})}h_{F}(\hat{x},x_{m})\cdot\psi_{F}\left(\left(x_{m}-c_{F}(\hat{x},\xi)\right)y_{m}\right)dx_{m}\right)d\hat{x}.\label{eq:1}
 \end{flalign}
Note that $h_{F}(\hat{x},x_{m})$ depends only on $\hat{x}$ when
$\val(x_{m}-c(\hat{x},\xi))=l$, since for any $j$
we have $\val(g_{j,F}(\hat{x},x_{m}))=\val(h_{j,F}(\hat{x},\xi))+w_{j}\cdot l$
and $\ac(g_{j,F}(\hat{x},x_{m}))$ does not depend on $x_{m}$. Therefore
after a linear change of variables $u:=x_{m}-c(\hat{x},\xi)$ we can
write (\ref{eq:1}) as: 
\begin{equation*}
\sum_{\xi\in k_{F}^{r}}\sum_{l\in\mathbb{Z}}\int_{W_{F}(l,\xi)}\psi_{F}(\langle\hat{y},\hat{x}\rangle+c_{F}(\hat{x},\xi)\cdot y_{m})\cdot h_{F}(\hat{x},x_{m})\cdot\left(\int_{B_{F}(l,\xi,0,\hat{x})}\psi_{F}\left(u\cdot y_{m}\right)du\right)d\hat{x}.
\end{equation*}
Now, for every $\hat{x}\in F^{m-1}$, we have $B_{F}(l,\xi,0,\hat{x})=[\xi_{1}]\pi^{l}+\pi^{l+1}\mathcal{O}_{F}$,
where $[\xi_{1}]\in\mathcal{O}_{F}^{\times}$ is the unique lift of
$\xi_{1}\in k_{F}$ to $\mathcal{O}_{F}$. If we take $y=(\hat{y},y_{m})$
with $y_{m}$ such that $\val(y_{m})\leq-\val(u)-2=-l-2$, then for
$j=-(l+\val(y_{m})+1)\geq1$ we have 
\[
\{u\cdot y_{m}:u\in B_{F}(l,\xi,0,\hat{x})\}=a\cdot\pi^{-1-j}+\pi^{-j}\mathcal{O}_{F},
\]
for some $a\in\mathcal{O}_{F}^{\times}$. Set $uy_{m}=a\pi^{-1-j}+z\pi^{-j}$
where $z\in\mathcal{O}_{F}$ depends on $u$, and recall that $\psi_{F}(\mathcal{O}_{F})=1$.
For any $y_{m}$ with $\val(y_{m})\leq-\val(u)-2$, we obtain $\int_{B_{F}(l,\xi,0,\hat{x})}\psi_{F}\left(u\cdot y_{m}\right)du=0,$
since this integral is essentially a sum of a non-trivial character
over a finite group: 
\begin{align*}
\int_{B_{F}(l,\xi,0,\hat{x})}\psi_{F}\left(u\cdot y_{m}\right)du & =\psi_{F}\left(a\cdot\pi^{-1-j}\right)\cdot\int_{\pi^{l+1}\mathcal{O}_{F}}\psi_{F}\left(z\cdot y_{m}\right)dz\\
 & =\psi_{F}\left(a\cdot\pi^{-1-j}\right)\cdot\left|\frac{1}{y_{m}}\right|_{F}\cdot\int_{\pi^{-j}\mathcal{O}_{F}}\psi_{F}\left(z\right)dz\\
 & =\psi_{F}\left(a\cdot\pi^{-1-j}\right)\cdot\left|\frac{1}{y_{m}}\right|_{F}\cdot\sum_{\widetilde{z}\in\pi^{-j}\mathcal{O}_{F}/\mathcal{O}_{F}}\widetilde{\psi}_{F}(\widetilde{z})=0,
\end{align*}
where $\widetilde{\psi}_{F}$ is the character induced on $\pi^{-j}\mathcal{O}_{F}/\mathcal{O}_{F}$
from $\psi_{F}$.

Now, for $a\in\VF$ and $\xi\in\RF^{r}$ define $B'(a,\xi)=\{(\hat{x},x_{m})\in A(\xi):\val(a\cdot(x_{m}-c(\hat{x},\xi)))\geq-1\}$
and $B'(a)=\bigcup\limits _{\xi\in\RF^{r}}B'(a,\xi)$. By the above
computation, 
\[
\int_{A_{F}\backslash B'_{F}(y_{m})}h_{F}(\hat{x},x_{m})\psi_{F}(\langle y,x\rangle)dx=0.
\]
We get, 
\begin{eqnarray*}
\left|\int_{A_{F}}h_{F}(\hat{x},x_{m})\psi_{F}(\langle y,x\rangle)dx\right| & = & \left|\int_{B'_{F}(y_{m})}h_{F}(\hat{x},x_{m})\psi_{F}(\langle y,x\rangle)dx\right|\leq\int_{B'_{F}(y_{m})}\left|h_{F}(\hat{x},x_{m})\right|dx.
\end{eqnarray*}
Note that $\left|h\right|$ is a motivic function, and that $\left|h_{F}\right|$
is integrable on $B'_{F}(y_{m})$ for any $F\in Loc_{M'}$ and any
$y_{m}\in F$. Recall we assumed there exists a definable set $L$
such that $\supp(h_{F})\subset L_{F}$ for all $F\in\mathrm{Loc}_{M'}$.
By Theorem \ref{integration of motivic}, there exist a motivic function
$g^{A}\in\mathcal{C}(\VF)$ and an integer $M>M'$ such that for any
$F\in\mathrm{Loc}_{M}$ we have, 
\[
\int_{B'_{F}(y_{m})}\left|h_{F}(\hat{x},x_{m})\right|dx=\int_{B'_{F}(y_{m})\cap L_{F}}\left|h_{F}(\hat{x},x_{m})\right|dx=g_{F}^{A}(y_{m}).
\]
This implies, 
\[
\left|\int_{A^{F}}h_{F}(\hat{x},x_{m})\psi_{F}(y\cdot x)\right|\leq g_{F}^{A}(y_{m}).
\]
Now, for any $F\in Loc_{M}$, the set $L_{F}$ is compact, and since
the measure of $L_{F}\cap B_{F}'(y_{m})$ tends to zero as $\left|y_{m}\right|$
tends to infinity, we get that $\lim\limits _{\left|y_{m}\right|\rightarrow\infty}g_{F}^{A}(y_{m})=0$.
By repeating these arguments for the other cells, and possibly enlarging
$M$, we obtain that for any $F\in Loc_{M}$, 
\[
\left|\mathcal{F}(h_{F})(y)\right|\leq\sum_{i=1}^{N}\left|\int_{A_{i,F}}h_{F}(\hat{x},x_{m})\psi(\langle x,y\rangle)dx\right|\leq\sum_{i=1}^{N}g_{F}^{A_{i}}(y_{m}),
\]
where $\lim\limits _{\left|y_{m}\right|\rightarrow\infty}\left(\sum_{i=1}^{N}g_{F}^{A_{i}}(y_{m})\right)=0$,
as required. 
\end{proof}
The following proposition, along with Lemma \ref{lemma 5.4}, finishes
the proof of Theorem \ref{thm:L^1 constructible has polynomial decaying Fourier transform}.
Indeed, for $g\in\mathcal{C}(\VF)$ as in Lemma \ref{lemma 5.4},
Proposition \ref{One dimensional lemma} yields constants $d$ and
$\alpha$ such that, 
\[
|\mathcal{F}(h_{F})(y_{1},{\ldots},y_{m})|\leq g_{F}(y_{m})\leq d(F)\cdot\mathrm{min}\{\left|y_{m}\right|{}^{\alpha},1\}.
\]
\begin{prop}
\label{One dimensional lemma} Let $f\in\mathcal{C}(\VF)$ and suppose
that $\lim\limits _{\left|y\right|\to\infty}f_{F}(y)=0$ for any $F\in\mathrm{Loc}_{M}$.
Then there exists a real number $\alpha<0$ such that for any $F\in\mathrm{Loc}_{M}$,
\[
\left|f_{F}(y)\right|<d(F)\cdot\mathrm{min}\{\left|y\right|{}^{\alpha},1\},
\]
where $d(F)>0$ is a constant that depends only on $F$. 
\end{prop}

Recall that $f_{F}$ is of the following form, where $\alpha_{i}$,$\beta_{ij}$
and $Y_{i}$ are as in Definition \ref{def:motivic function}: 
\[
f_{F}(y)=\sum_{i=1}^{N}|Y_{i,F,y}|q_{F}^{\alpha_{i,F}(y)}\cdot\left(\prod_{j=1}^{N'}\beta_{ij,F}(y)\right)\left(\prod_{l=1}^{N''}\frac{1}{1-q_{F}^{a_{il}}}\right).
\]
By quantifier elimination (Theorem \ref{Quantifier elimination theorem}),
there exist polynomials $\{g_{l}\}_{l=1}^{s}\in\mathbb{Z}[y]$ such
that (recall Notation \ref{nota:Given-an--formula}): 
\begin{enumerate}
\item $\alpha_{i}$ has datum $\Xi(\alpha_{i})=(\{g_{l}\}_{l=1}^{s},\{\chi_{k}^{i}(y)\}_{k=1}^{N_{i}},\{\theta_{k}^{i}(y,t)\}_{k=1}^{N_{i}})$,
with $t\in{\ints}$, $y\in\VF$. 
\item $\beta_{ij}$ has datum $\Xi(\beta_{ij})=(\{g_{l}\}_{l=1}^{s},\{\chi_{k}^{ij}(y)\}_{k=1}^{N_{ij}},\{\theta_{k}^{ij}(y,t)\}_{k=1}^{N_{ij}})$,
with $t\in{\ints}$, $y\in\VF$. 
\item $Y_{i}\subseteq\VF\times\mathrm{RF}^{r_{i}}$ has datum $\Xi(Y_{i})=(\{g_{l}\}_{l=1}^{s},\{\widetilde{\chi}_{k}^{i}(y,\eta)\}_{k=1}^{M_{i}},\{\widetilde{\theta}_{k}^{i}(y)\}_{k=1}^{M_{i}})$,
with $r_{i}\in\nats$, $\eta\in\mathrm{RF}^{r_{i}}$, $y\in\VF$. 
\end{enumerate}
By the uniform cell decomposition theorem, we can decompose $\VF$
as a finite disjoint union of cells, where each cell is of the form $A=\bigcup\limits _{\xi\in\RF^{r}}A(\xi)$
with datum $\Theta(A)=(C,b_{1}(\xi),b_{2}(\xi),c(\xi),\lambda)$
for a definable set $C\subseteq\RF^{r}$ and definable functions $b_{1},b_{2}$
and $c$ from $C$ to $\VF$. Moreover, we have $\val(g_{l}(y))=\val(h_{l}(\xi)(y-c(\xi))^{w_{l}})$
and $\ac(g_{l}(y))=\xi_{\mu(l)}$ for all $1\leq l\leq s$.

Since there are finitely many cells, it is enough to prove Proposition
\ref{One dimensional lemma} for $f|_{A}$ and any cell $A$. We do
the latter by proving several smaller lemmas. 
\begin{lem}
\label{Lemma 5.6} We may assume that $\val(g_{l}(y))=w_{l}\cdot\val(y)+\val(h_{l}(\xi))$
for all $1\leq l\leq s$. 
\end{lem}

\begin{proof}
For a given cell $A$, consider the definable sets $L(\xi):=\{y\in A(\xi):\val(c(\xi))>\val(y)\}$
and $L=\bigcup\limits _{\xi\in\RF^{r}}L(\xi)$. Since we are interested
in asymptotic behavior (i.e when $\val(y)\rightarrow-\infty$), it
is enough to prove the claim for $f|_{L}$ and each such $L$, but
notice that for any $y\in L$ we have 
\[
\val(g_{l}(y))=\val(h_{l}(\xi)(y-c(\xi))^{w_{l}})=w_{l}\cdot\val(y)+\val(h_{l}(\xi)).
\]
\end{proof}
\begin{lem}
\label{Lemma 5.7} Each cell $A$ has a partition $A=\bigcup\limits _{b=1}^{N_{A}}{A_{b}}$
such that on each $A_{b}$ we have the following: 
\end{lem}

\begin{enumerate}
\item The functions $\beta_{ij}$ and $\alpha_{i}$ can be written as compositions
\[
\alpha_{i}(y)=\widetilde{\alpha_{i}}\circ\left(\val(g_{1}(y)),\ldots,\val(g_{s}(y))\right)\text{ and }\beta_{ij}(y)=\widetilde{\beta}_{ij}\circ\left(\val(g_{1}(y)),\ldots,\val(g_{s}(y))\right),
\]
where $\widetilde{\alpha}_{i}$ and $\widetilde{\beta_{ij}}$ are
$\mathcal{L}_{\mathrm{Pres}}^\infty$-definable functions from $\mathcal{L}_{\mathrm{Pres}}^\infty$-definable
subsets of $\ints^{s}$ to $\ints$. 
\item $|Y_{i,F,y}|$ depends only on $\xi$ (and $F$). In particular, for any $\xi\in k_{F}^{r}$,
the value $|Y_{iF}(y)|$ is constant on $A_{b}(\xi)$. 
\end{enumerate}
\begin{proof}
We prove for $\alpha_{i}$, the proof for $\beta_{ij}$ is similar.
We start with constructing a partition of $A$ which satisfies (1).
Recall that $\alpha_{i}$ is defined by the formula 
\[
\psi_{\alpha_{i}}(y,t)=\bigvee\limits _{k=1}^{N_{i}}\chi_{k}^{i}(\ac(g_{1}(y)),\ldots,\ac(g_{s}(y)))\wedge\theta_{k}^{i}(\val(g_{1}(y)),\ldots,\val(g_{s}(y)),t).
\]
For simplicity we set $\ac(g(y)):=\ac(g_{1}(y)),\ldots,\ac(g_{s}(y)))$,
and use similar notation for $\val(g(y))$. Then $\alpha_{i}^{-1}(l')=\bigcup\limits _{k=1}^{N_{i}}\{y:\chi_{k}^{i}(\ac(g(y)))\wedge\theta_{k}^{i}(\val(g(y)),l')\}$
for any $l'\in\ints$. Let $I\in\{0,1\}^{N_{i}}$ and consider the
set 
\[
A_{I}=\{y\in A:\chi_{k}^{i}(\ac(g(y)))=I(k),~1\leq k\leq N_{i}\}.
\]
If we restrict to $A_{I}$, we get that $\alpha_{i}(y)=l'$ if and
only if $\sigma_{i}(y,l'):=\bigvee\limits _{1\leq k\leq N_{i},I(k)=1}\theta_{k}^{i}(\val(g(y)),l')$
is true. Since $\alpha_{i}$ is a definable function, for any $y\in A_{I}$
there is at most one $l'$ such that $\sigma_{i}(y,l')$ holds, and
thus $\sigma_{i}(y,l')$ is a graph of a definable function. Now,
note that $\widetilde{\sigma_{i}}(z):=\bigvee\limits _{1\leq k\leq N_{i},I(k)=1}\theta_{k}^{i}(z)$
is an $\mathcal{L}_{\mathrm{Pres}}^\infty$-formula in $\ints^{s+1}$. Thus,
it is a graph of a definable function into $\ints$ when restricted to
the $\mathcal{L}_{\mathrm{Pres}}^\infty$-definable set $\{z\in\ints^{s}:\exists!z'\text{ s.t. }\widetilde{\sigma_{i}}(z,z')\text{ holds}\}$.
Since $\sigma_{i}(y,l')=\widetilde{\sigma_{i}}\left(\val(g(y)),l'\right)$,
we get that $\alpha_{i}|_{A_{I}}$ is of the required form.

To prove the second property, recall that $Y_{i}$ is defined by the
formula 
\[
\psi_{Y_{i}}=\bigvee\limits _{k=1}^{M_{i}}\widetilde{\chi}_{k}^{i}(\ac(g(y)),\eta)\wedge\widetilde{\theta}_{k}^{i}(\val(g(y))).
\]
By a process similar to before, we can find a partition $A=\bigcup\limits _{b=1}^{\widetilde{N}_{A}}\widetilde{A}_{b}$
such that for every $b$ each $\widetilde{\theta}_{k}^{i}(\val(g(y)))$
is either identically true on $\widetilde{A}_{b}$ or identically
false. Moreover, on each $\widetilde{A}_{b}\subseteq A$, since $\ac(g_{l}(y))=\xi_{\mu(l)}$
for any $l$, the formula $\widetilde{\chi}_{k}^{i}(\ac(g(y)),\eta)$
depends only on $\xi$ and $\eta$. This implies that the assignment
$y\mapsto Y_{i,F,y}=\{\eta\in k_{F}^{r_{i}}:(\eta,y)\in Y_{i,F}\}$
is constant on $\widetilde{A}_{b}(\xi)$ for any $F\in\mathrm{Loc}_{M}$,
and thus $|Y_{i,F}(y)|$ is constant on $\widetilde{A}_{b}(\xi)$ and
depends only on $\xi$ in $\widetilde{A}_{b}$, as required. To finish
the proof, we take $\{A_{b}\}_{b=1}^{N_{A}}$ to be the joint refinement
of $\{\widetilde{A}_{b}\}_{b=1}^{\widetilde{N}_{A}}$ and $\{A_{I}\}$
as constructed above. 
\end{proof}
Recall we are trying to bound the decay rate of a motivic function
$f\in\mathcal{C}(\VF)$ which tends to zero at infinity. The following
lemma allows us, after further refining our cover, to give a simple
description of $f$ on each piece. 
\begin{lem}
\label{Lemma 5.8} By further refining the cells $\curly{A_{b}}_{b=1}^{N_{A}}$,
we can assume the restriction of $f$ to each cell is of the form
\[
f_{F}(y)=\sum_{i=1}^{N}a_{i}(\xi)\cdot q_{F}^{t_{1i}(\xi)+t_{2i}\cdot\val(y)}\cdot P_{i,\xi}(\val(y))\cdot Q_{i}(q_{F}),
\]
where $a_{i}(\xi)$ are real numbers, $t_{1i}(\xi)$ and $t_{2i}$
are rational numbers, $P_{i,\xi}\in\rats[x]$ are polynomials with
coefficients that may depend on $\xi$, and $Q_{i}(q_{F})$ is an
expression of the form $\prod\limits _{l=1}^{N'}\frac{1}{1-q_{F}^{a_{il}}}$. 
\end{lem}

\begin{proof}
Let $L(\xi)=\{y\in A(\xi):\val(c(\xi))>\val(y)\}$ and $L=\bigcup\limits _{\xi\in\RF^{r}}L(\xi)$
as in the proof of Lemma \ref{Lemma 5.6}, set $1\leq b\leq N_{A}$
and consider $A_{b}$, $\widetilde{\alpha_{i}}$ and $\widetilde{\beta}_{ij}$,
as defined in Lemma \ref{Lemma 5.7}. By the Presburger cell decomposition
(Theorem \ref{Presburger Cell decomposition}), $\widetilde{\alpha_{i}}$
and $\widetilde{\beta}_{ij}$ are piecewise linear in the sense of
Definition \ref{def:-Linear function}. Thus, we have a finite partition
$\curly{X_{b'}}$ of $\curly{(\val(g_{1}(y)),\ldots,\val(g_{s}(y))):y\in A_{b}}\subseteq \ints^s$
such that on each $X_{b'}$ the functions $\widetilde{\alpha_{i}}$
and $\widetilde{\beta}_{ij}$ are linear. Pulling $\curly{X_{b'}}$
back to $A_{b}$ under $\val(g(-))$, we then obtain a partition $A_{b}=\bigcup\limits _{b'=1}^{N'_{A_{b}}}A'_{b'}$.
By Lemmas \ref{Lemma 5.6} and \ref{Lemma 5.7}, and the linearity of $\widetilde{\alpha_{i}}$
and $\widetilde{\beta}_{ij}$ on $A'_{b'}$, the restriction $f_{F}|_{A'_{b',F}\cap L_{F}}$
is a sum of terms of the following form: 
\[
|Y_{F,y}|\cdot q_{F}^{\rho}\cdot\left(\prod_{j=1}^{N}\rho_{j}\right)\cdot\left(\prod_{l=1}^{N'}\frac{1}{1-q_{F}^{a_{l}}}\right),
\]
where 
\begin{enumerate}
\item $\rho=\sum\limits _{t=1}^{s}b_{t}\cdot\frac{w_{t}\cdot\val(y)+\val(h_{t}(\xi))-c_{t}}{n_{t}}+\gamma$. 
\item $\rho_{j}=\sum\limits _{t=1}^{s}b_{jt}\cdot\frac{w_{t}\cdot\val(y)+\val(h_{t}(\xi))-c_{jt}}{n_{jt}}+\gamma_{j}$. 
\item The constants $b_{t},n_{t},c_{t},\gamma,b_{jt},n_{jt},c_{jt}$ and
$\gamma_{j}$ are integers. 
\item It holds that $0\leq c_{t}<n_{t}$ and $0\leq c_{jt}<n_{jt}$ for
all $1\leq t\leq s$ and $1\leq j\leq N$ and furthermore, 
\[
w_{t}\cdot\val(y)+\val(h_{t}(\xi))\equiv c_{t}(\mathrm{mod}n_{t})\text{ and }w_{t}\cdot\val(y)+\val(h_{t}(\xi))\equiv c_{jt}(\mathrm{mod}n_{jt}).
\]
\end{enumerate}
Now for a fixed $F\in\mathrm{Loc}_{M}$ and $\xi\in k_{F}^{r}$, the
functions $|Y_{F,y}|$ and $h_{t}(\xi)$ are constant on $L_{F}(\xi)\cap A'_{b',F}(\xi)$.
This gives $f_{F}|_{L_{F}\cap A'_{b',F}}$ the required form. 
\end{proof}
We are ready to prove Proposition \ref{One dimensional lemma}: 
\begin{proof}[Proof of Proposition \ref{One dimensional lemma}]
Fix $F\in\mathrm{Loc}_{M}$ and $\xi\in k_{F}^{r}$. By Lemma \ref{Lemma 5.8},
we have a partition of $\VF$ into cells, such that on each cell $\widetilde{A}$
the function $f_{F}$ is of the form 
\[
f_{F}(y)=\sum_{i=1}^{N}a_{i}\cdot q_{F}^{t_{1i}+t_{2i}\cdot\val(y)}\cdot P_{i}(\val(y))Q_{i}(q_{F}),
\]
for any $y\in\widetilde{A}$ with $\val(c(\xi))>\val(y)$. Denote
the above summands of $f_{F}$ by $f_{i}$. If $t=\max\limits _{i}\{t_{2i}\}>0$,
consider the sum $\underset{i:t_{2i}=t}{\sum}f_{i}$ of all terms
$f_{i}$ such that $t_{2i}=t$. We want to show this sum is zero.
Since $q_{F}^{t\cdot\val(y)}\neq0$ 
it is enough to prove 
\[
\underset{i:t_{2i}=t}{\sum}a_{i}\cdot q_{F}^{t_{1i}}\cdot P_{i}(\val(y))Q_{i}(q_{F})=0.
\]
Note that the above sum is polynomial in $\val(y)$, and denote it
by $P'_{t}(\val(y))$. If $P'_{t}\not\equiv0$, then we have $\lim\limits _{\mathrm{val}(y)\rightarrow-\infty}\left|P'_{t}(\val(y))\right|=\infty$.
But then $\lim\limits _{\left|y\right|\to\infty}\frac{\left|f_{F}(y)\right|}{q_{F}^{t\cdot\val(y)}}=\infty$
and thus $\lim\limits _{\left|y\right|\to\infty}f_{F}(y)\neq0$, yielding
a contradiction. Thus, we can write $f_F$ in a reduced way, where $\max\limits _{i}\{t_{2i}\}<0$.
Let $\alpha_{\widetilde{A}}:=\frac{1}{2}\max\limits _{i}\{t_{2i}\}.$
The above argument implies that for any $F\in\mathrm{Loc}_{M}$, any
$\xi\in k_{F}^{r}$, and any $y\in\widetilde{A}$ we have: 
\[
\left|f_{F}(y)\right|<d(F,\xi)\cdot\mathrm{min}\{\left|y\right|^{\alpha_{\widetilde{A}}},1\},
\]
for some constant $d(F,\xi)>0$ that depends on $\xi$ and $F$. Since
there are finitely many cells $\widetilde{A}$ (and fibers $\widetilde{A}(\xi)$),
by taking the maximum over all $\alpha_{\widetilde{A}}$ above, 
 there exists $\alpha<0$ such that for any $F\in\mathrm{Loc}_{M}$
and $y\in F$, \textcolor{red}{{} } 
\[
\left|f_{F}(y)\right|<d(F)\cdot\mathrm{min}\{\left|y\right|{}^{\alpha},1\}
\]
for some constant $d(F)>0$. 
\end{proof}

\subsection{\label{subsec:Alternative-proof-of Theorem 5.2}Alternative proof
of Theorem \ref{thm:L^1 constructible has polynomial decaying Fourier transform}}
In this subsection we give another proof to Theorem \ref{thm:L^1 constructible has polynomial decaying Fourier transform}
as suggested by the anonymous referee. The main idea is to study the
Fourier transform $\mathcal{F}(h)$ of an $L^{1}$-motivic function
$h$ as a special case of a motivic exponential function (Definition
\ref{def:motivic exponential function}) which decays to $0$ at infinity,
rather than to study the decay of $\mathcal{F}(h)$ directly. 

For $F\in\mathrm{Loc}$, we denote by $\mathcal{D}_{F}$ the collection
of additive characters of $F$ that are trivial on $\mathfrak{m}_{F}$,
and satisfy $\psi_{F}(x)=\mathrm{exp}{}^{\frac{2\pi i}{p}\cdot\mathrm{Tr}_{k_{F}/\mathbb{F}_{p}}(\overline{x})}$
for any $x\in\mathcal{O}_{F}$, where $\mathrm{char}(k_{F})=p$. 
\begin{defn}[{see \cite[Definition 2.7]{CGH16}}]
\label{def:motivic exponential function} Let $X$ be an $\Ldp$-definable
set. A collection of functions $f=\{f_{F,\psi}:X_{F}\rightarrow\complex\}_{F\in\mathrm{Loc}_{M},\psi\in\mathcal{D}_{F}}$
is called a \textit{motivic exponential function} if there exist $N\in\nats$ and non-negative integers
$\{r_{i}\}_{i=1}^{N}$, such that for any $F\in\mathrm{Loc}_{M}$, any
$\psi\in\mathcal{D}_{F}$, and any $x\in X_{F}$, the function $f$ can be written
as,
\[
f_{F,\psi}(x)=\sum_{i=1}^{N}f_{i,F}(x)\cdot\left(\sum_{y\in Y_{i,x,F}}\psi\left(g_{i,F}(x,y)+e_{i,F}(x,y)\right)\right),
\]
where $f_{i}\in\mathcal{C}(X)$, $Y_{i}\subseteq X\times\mathrm{RF}^{r_{i}}$
are definable sets, and $e_{i}:Y_{i}\rightarrow\RF$ and $g_{i}:Y_{i}\rightarrow\VF$
are definable functions for any $1\leq i \leq N$ (we make sense of
the expression inside $\psi$ by lifting the values of $e_{i,F}(x,y)$ to $\mathcal{O}_{F}^{\times}$). The set of motivic exponential functions on
a definable set $X$ forms a ring, which we denote by $\mathcal{C}_{\mathrm{exp}}(X)$. 

The ring $\mathcal{C}_{\mathrm{exp}}(X)$ is closed under integration
(see \cite[Theorem 2.8]{CGH16}), and in particular, the Fourier transform
$\mathcal{F}(f)$ of an $L^{1}$-function $f\in\mathcal{C}(X)$ belongs
to $\mathcal{C}_{\mathrm{exp}}(X)$. We now turn to the proof of Theorem
\ref{thm:L^1 constructible has polynomial decaying Fourier transform}.
\end{defn}

\begin{proof}[Alternative proof of Theorem \ref{thm:L^1 constructible has polynomial decaying Fourier transform}]
 Let $h\in\mathcal{C}(\VF^{m})$, and assume that $h_{F}$ is absolutely
integrable for any $F\in\mathrm{Loc}_{M}$. We may drop the assumption
that $h_{F}$ is compactly supported. 
Set $\hat{h}(y):=\mathcal{F}(h)(y)\in\mathcal{C}_{\mathrm{exp}}(\VF^{m})$.
For any $F\in\mathrm{Loc}_{M}$ and $\psi\in\mathcal{D}_{F}$ the function 
$\left|\hat{h}_{F,\psi}(y)\right|$ is bounded on $F^{m}$, 
and by the Riemann-Lebesgue
lemma it decays
to $0$ as $\left|y\right|\rightarrow\infty$. 
We can divide $\mathrm{VF}^{m}$ to $2^{m}$ definable subsets 
\[
A_{I}:=\{y\in\mathrm{VF}^{m}: 1 \leq i \leq m,\,\mathrm{val}(y_{i})\geq0\text{ if }I(i)=1
\text{ and }\mathrm{val}(y_{i})<0\text{ if }I(i)=0\},
\]
with $I\in\{0,1\}^{m}$. It is enough to prove the claim for $\hat{h}|_{A_{I}}$
for each such $I$. Fix $I$, and set $n_{I}=\#\{i:I(i)=1\}$. Without loss of generality,
assume that $I(i)=1$ for $i=1,\ldots,n_{I}$ and $I(i)=0$ otherwise.
For $t \in \ints^m$,  set $t_{ \geq 0}:=(t_{1},\ldots,t_{n_{I}})$ and $t_{<0}:=(t_{n_{I}+1},\ldots,t_{m})$,
and define $H\in\mathcal{C}_{\mathrm{exp}}(\mathrm{VF}^{m}\times\mathbb{Z}^{m})$
by $H(y,t_{\geq 0},t_{<0}):=\hat{h}(y)\cdot1_{B(t_{\geq 0},t_{<0})}$, where
\[
B(t_{\geq 0},t_{<0}):=\{y\in A_{I}: \mathrm{val}(y_{i})=t_{i}\text{ for } 1 \leq i \leq m\}.
\]
Set $W=\mathrm{VF}^{m}\times\mathbb{Z}^{n_{I}}$ and $X=\mathbb{Z}^{m-n_{I}}$.
By \cite[Theorem 2.1.3]{CGH18}, we can find $G\in\mathcal{C}_{\mathrm{exp}}(\mathbb{Z}^{m-n_{I}})$
and $k\in\mathbb{Z}$ such that for each $F\in\mathrm{Loc}_{M'}$
with $0\ll M'\in\nats$, and each $\psi\in\mathcal{D}_{F}$ the following holds:
\begin{equation}
\underset{(y,t_{\geq 0})\in W_{F}}{\mathrm{sup}}\left|H_{F,\psi}(y,t_{\geq 0},t_{<0})\right|^{2}=\underset{y\in B_{t_{<0},F}}{\mathrm{sup}}\left|\hat{h}_{F,\psi}(y)\right|^{2}\leq G_{F,\psi}(t_{<0})\leq q_{F}^{k}\underset{(y,t_{\geq 0})\in W_{F}}{\mathrm{sup}}\left|H_{F,\psi}(y,t_{\geq 0},t_{<0})\right|^{2},\label{eq:5.3}
\end{equation}
where $B_{t_{<0}}:=\bigcup\limits_{t_{\geq 0}\in \ints^{n_I}} B(t_{\geq 0},t_{<0})$.
Note that $G_{F,\psi}(t_{<0})$ is bounded. 
Set $\left\Vert t_{<0}\right\Vert :=\sum\limits_{i=n_{I}+1}^{m}\left|t_{i}\right|$,
then it is clear from the definition of $H$ that 
$H_{F,\psi}(y,t_{\geq 0},t_{<0})=0$
whenever $t_{i}\geq0$ for some $n_I+1 \leq i \leq m $. 
It then follows 
from (\ref{eq:5.3}) that $\underset{\left\Vert t_{<0}\right\Vert \rightarrow\infty}{\mathrm{lim}}\left|G_{F,\psi}(t_{<0})\right|=0$.
By \cite[Theorem 3.1.1(2)]{CGH18}, we can find $r\in\mathbb{Q}_{>0}$
and $a\in\mathbb{Z}$ such that for any $F\in\mathrm{Loc}_{M'}$  (we
may assume that $M'$ is large enough) and each $\psi\in\mathcal{D}_{F}$
we have
\[
\left|G_{F,\psi}(t_{<0})\right|\leq q_{F}^{r\left\Vert t_{<0}\right\Vert },
\]
for any $t_{<0}$ with $\left\Vert t_{<0}\right\Vert >a$. We then deduce from
(\ref{eq:5.3}) 
\[
\left|\hat{h}_{F,\psi}(y)\right|^{2}<q_{F}^{r\sum_{i=n_{I}+1}^{m}\left|\mathrm{val}(y_{i})\right|}= q_{F}^{-r\sum_{i=n_{I}+1}^{m}\mathrm{val}(y_{i})}\leq\left|y\right|^{r(m-n_{I})}
\]
 for any $y\in A_{I}$ such that $\left|y\right|>q^{a}$. To deal
with small $y$, we use \cite[Theorem 2.1.1]{CGH18}; there exist integers
$b$ and $d'$ such that for any $F\in\mathrm{Loc}_{M'}$,
any $\psi\in\mathcal{D}_{F}$ and any $t_{<0}$ we have 
\[
\underset{y\in B_{t_{<0},F}}{\mathrm{sup}}\left|\hat{h}_{F,\psi}(y)\right|^{2}\leq\left|G_{F,\psi}(t_{<0})\right|\leq q_{F}^{b\left\Vert t_{<0}\right\Vert +d'}.
\]
Repeating the above argument for each $A_{I}$, and setting
$a_{0}$ (resp. $b_{0},d'_{0},\alpha$) to be maximal amongst the $a$'s
(resp. $b,d',\frac{r(m-n_{I})}{4}$) above, over all possible $I$, we get
\begin{equation}
\left|\mathcal{F}(f)_{F,\psi}(y)\right|<q_{F}^{\frac{1}{2}(b_{0}a_{0}+d'_{0})}\cdot\mathrm{min}\{\left|y\right|{}^{\alpha},1\}\label{eq:(5.4)}
\end{equation}
as required.
\end{proof}
\begin{rem}
Since the results in \cite{CGH18} used above work also for
definable families of motivic exponential functions, one can
provide a version of Theorem \ref{thm:L^1 constructible has polynomial decaying Fourier transform}
for families. This can be used to prove part (2) of Theorem \ref{generalization of the main result},
for the case that $K=\rats$. For a general field of characteristic
zero, one still needs to use arguments similar to those of Sections \ref{sec:Reduction-to-k=00003D00003D00003D00003DQ}
and \ref{sec:Convolution-properties-of}.
\end{rem}

\section{Reduction of Theorem \ref{Main result} to the case of $\rats$-morphisms}

\label{sec:Reduction-to-k=00003D00003D00003D00003DQ}

In this section we complete the proof of Theorem \ref{Main result},
by reducing to the case $K=\rats$ which we have proved in the last
section (Theorem \ref{Main result-Theorem 5.1}). Explicitly, we show
the following: 
\begin{prop}
\label{reduction to dominant} It is enough to prove Theorem \ref{Main result}
for $K=\rats$ and $\varphi:X\rightarrow V$ strongly dominant. 
\end{prop}

We prove the claim by a series of reductions: 
\begin{lem}
\label{lemma 6.2} It is enough to prove Theorem \ref{Main result}
for a field $K'$ which is finitely generated over $\rats$, and $\varphi:X\rightarrow V$
a strongly dominant $K'$-morphism. 
\end{lem}

\begin{proof}
Let $K$ be a field of characteristic $0$. We have already seen in
Corollary \ref{Cor-reduction to dominant} that we may assume that
$\varphi$ is strongly dominant. Notice that since $\rats\subseteq K$,
we have that $\varphi$ is defined over a finitely generated field
$K'/\rats$. Since the (FRS) property is preserved under base change
(Proposition \ref{FRS is a good property}), we are done. 
\end{proof}
\begin{prop}
\label{prop:reduction to (FRS) at diagonal} Let $K'$ and $\varphi:X\rightarrow V$
be as in Lemma \ref{lemma 6.2}. Assume there exists $N\in\nats$
such that the $N$-th convolution power $\varphi^{N}$ is (FRS) at
$(x,{\ldots},x)$ for any $x\in X(\overline{K'})$. Then $\varphi^{2N}$
is (FRS). 
\end{prop}

\begin{proof}
We use the analytic criterion for the (FRS) property (see Theorem
\ref{Analytic condition for (FRS)}) to show that $\varphi^{2N}$
is (FRS) at any point $(x_{1},{\ldots},x_{2N})\in X^{2N}(\overline{K'})$.
Let $(x_{1},{\ldots},x_{2N})\in X^{2N}(\overline{K'})$, then there
exists a finite extension $K''/K'$ such that $(x_{1},{\ldots},x_{2N})\in X^{2N}(K'')$.
We need to show that for any finite extension $K'''$ of $K''$, there
exist $K'''\subseteq F\in\mathrm{Loc}$ and a non-negative Schwartz
measure $\mu$ on $X^{2N}(F)$ that does not vanish at $(x_{1},{\ldots},x_{2N})$,
such that $\varphi_{*}^{2N}(\mu)$ has continuous density.

Fix such $K'''$, and let $U_{i}\subset X^{N}$ be a Zariski open
neighborhoods of $(x_{i},{\ldots},x_{i})$ such that $\varphi^{N}$
is (FRS) at any $x\in U_{i}$ (it is possible by Theorem \ref{Lemma Elkik}).
Notice that for any $K'''\subseteq F\in\mathrm{Loc}$ the set
$U_{i}(F)$ contains a set of the form $V_{i,F}\times\dots\times V_{i,F}$,
where $V_{i,F}\subseteq X(F)$ is open and $x_{i}\in V_{i,F}$. For any $i\in\{1,{\ldots},2N\}$,
let $\mu_{i}$ be a non-negative Schwartz measure on $X(F)$ supported
on $V_{i,F}$, that does not vanish at $x_{i}$. Notice that $\supp(\mu_{i}\times\dots\times\mu_{i})\subseteq U_{i}(F)$.

Since $\varphi^{N}$ is (FRS) at $(x_{i},{\ldots},x_{i})$, the measure $\varphi_{*}^{N}(\mu_{i}\times\dots\times\mu_{i})$
has continuous density with respect to the normalized Haar measure
on $V(F)=F^{m}$. Now, use the standard identification between measures
and functions on a locally compact group, and recall that the Fourier
transform $\mathcal{F}(f)$ of an ${L}^{2}$-function $f$ on $F^{m}$
is an ${L}^{2}$-function and that the Fourier transform $\mathcal{F}(f)$
of an ${L}^{1}$-function $f:F^{m}\rightarrow\complex$ is a continuous
function. Since $\varphi_{*}^{N}(\mu_{i}\times\dots\times\mu_{i})$
is a compactly supported measure with continuous density, its density
is ${L}^{2}$, and hence 
\[
\mathcal{F}\left(\varphi_{*}^{N}(\mu_{i}\times\dots\times\mu_{i})\right)=\mathcal{F}\left(\varphi_{*}(\mu_{i})*\dots*\varphi_{*}(\mu_{i})\right)=\mathcal{F}(\varphi{}_{*}\mu_{i})^{N}
\]
has ${L}^{2}$-density. This implies $\mathcal{F}\left(\varphi{}_{*}\mu_{i}\right)$
has ${L}^{2N}$-density. By a generalization of H${\ddot{\text{o}}}$lder's
inequality, it follows that 
\[
\mathcal{F}\left(\left(\varphi{}_{*}\mu_{1}\right)*\dots*\left(\varphi_{*}\mu_{2N}\right)\right)=\prod_{i=1}^{2N}\mathcal{F}(\varphi_{*}\mu_{i})
\]
has ${L}^{1}$-density, and hence $\left(\varphi{}_{*}\mu_{1}\right)*\dots*\left(\varphi_{*}\mu_{2N}\right)$
has continuous density. As a consequence, we get that $\varphi^{2N}$
is (FRS) at $(x_{1},{\ldots},x_{2N})$. 
\end{proof}
Let $\varphi:X\rightarrow V$ be a strongly dominant $K'$-morphism.
Any finitely generated field $K'/\rats$ is a finite extension of
some $\rats(t_{1},{\ldots},t_{n})$. Since $X$ and $\varphi$ are
defined using finitely many polynomials with coefficients in $K'$,
there exists $f\in\rats[t_{1},{\ldots},t_{n}]$, such that $X$ and
$\varphi$ are defined over a ring $A$, which is finite over $\rats[t_{1},{\ldots},t_{n},f^{-1}]$.
Since $X$ is smooth over the generic point $\mathrm{Spec}(K')$,
by further localizing, we may assume that $X$ is smooth over $A$.
We denote the resulting $A$-model for the diagram $\varphi$ by 
$\varphi_{A}:X_{A}\rightarrow V_{A}$, i.e. $X,V$ and $\varphi$
are base changes of $X_{A},V_{A}$ and $\varphi_{A}$ to $\spec(K')$.

Since $A$ is finite type over $\rats$, the morphisms $X_{A}\rightarrow\mathrm{Spec}(A)\rightarrow\mathrm{Spec}(\rats)$
and $V_{A}\rightarrow\mathrm{Spec}(A)\rightarrow\mathrm{Spec}(\rats)$
endow $X_{A}$ and $V_{A}$ with a natural structure of $\rats$-varieties,
denoted by $\widetilde{X}_{A}$ and $\widetilde{V}_{A}$ (and we similarly
denote $\widetilde{\varphi}_{A}$). Hence, $\widetilde{X}_{A}$ (resp.
$\widetilde{V}_{A}$) is a family of smooth $\rats$-varieties (resp.
$\rats$-vector spaces) over $\mathrm{Spec(A)}$. Notice that the
convolution operation can be generalized to such $\mathrm{Spec}(A)$-families
by $\varphi*\varphi:=\mathrm{mult}_{A}\circ(\varphi,\varphi)$, where
$\mathrm{mult}_{A}:G\times_{\mathrm{Spec}(A)}G\rightarrow G$ is the
multiplication map over $\mathrm{Spec}(A)$, and $(\varphi,\varphi):X\times_{\mathrm{Spec}(A)}X\rightarrow G\times_{\mathrm{Spec}(A)}G$.
With the above terminology, we can now prove the following: 
\begin{lem}
\label{lemma 6.4} Let $\widetilde{X}_{a}$ and $\widetilde{V}_{a}$
be the fibers of the varieties $\widetilde{X}_{A}$ and $\widetilde{V}_{A}$
over $a\in\spec(A)(\overline{\rats})$. It is enough to show that
for every $a\in\spec(A)(\overline{\rats})$ there exists $n_{a}\in\nats$
such that $\widetilde{\varphi}_{a}^{n_{a}}:\widetilde{X}_{a}\times\dots\times\widetilde{X}_{a}\rightarrow\widetilde{V}_{a}$
is (FRS) at $y=(x,{\ldots},x)$ for any $x\in\widetilde{X}_{a}(\overline{\rats})$. 
\end{lem}

\begin{proof}
Notice that $\widetilde{\varphi}_{a}^{n}$ is the fiber of $\widetilde{\varphi}_{A}^{n}$
over $a\in\mathrm{Spec}(A)(\overline{\rats})$. Let $\pi:\widetilde{X}_{A}\rightarrow\mathrm{Spec}(A)$
be the $\rats$-structure map and consider $\pi^{n}:\widetilde{X}_{A}^{n}\rightarrow\mathrm{Spec}(A)$,
where $\widetilde{X}_{A}^{n}:=\widetilde{X}_{A}\times_{\spec(A)}{\ldots}\times_{\spec(A)}\widetilde{X}_{A}$.
It follows by \cite[Corollary 2.2]{AA16}, that the set 
\[
U_{n}:=\{x\in\widetilde{X}_{A}(\overline{\rats}):\widetilde{\varphi}_{\pi^{n}(y)}^{n}:\widetilde{X}_{\pi^{n}(y)}^{n}\rightarrow\widetilde{V}_{\pi^{n}(y)}\text{ is (FRS) at }y=(x,{\ldots},x)\}
\]
is open for any $n\in\nats$. By our assumption, $\widetilde{X}_{A}(\overline{\rats})=\bigcup\limits _{n=1}^{\infty}U_{n}$,
and by the fact that the $\curly{U_{n}}_{n=1}^{\infty}$ are increasing
combined with quasi-compactness, there exists $N\in\nats$ such that
$U_{N}=\widetilde{X}_{A}$, implying that $\widetilde{\varphi}_{A}^{N}$
is (FRS) at $(x,{\ldots},x)$ for any $x\in\widetilde{X}_{A}(\overline{\rats})$.
By the last proposition, the morphism $\widetilde{\varphi}_{A}^{2N}$
is (FRS). Since $\varphi^{2N}:X^{2N}\rightarrow V$ is the generic
fiber of $\widetilde{\varphi}_{A}^{2N}$, then $\varphi^{2N}$ is
(FRS) as well. 
\end{proof}
The next lemma is the final reduction before we can prove Proposition
\ref{reduction to dominant}. 
\begin{lem}
It is enough to prove Theorem \ref{Main result} for a number field
$K$ and $\varphi:X\rightarrow V$ a strongly dominant $K$-morphism. 
\end{lem}

\begin{proof}
Let $a\in\spec(A)(\overline{\mathbb{Q}})$. Then there exists a finite
extension $K/\rats$ such that $a\in\spec(A)(K)$. By our assumption,
there exists $n\in\nats$ such that the morphism $\widetilde{\varphi}_{a}^{n}:\widetilde{X}_{a}\times\dots\times\widetilde{X}_{a}\rightarrow\widetilde{V}_{a}$
is (FRS) and by Lemma \ref{lemma 6.4} we are done. 
\end{proof}
We can now finish the proof of Proposition \ref{reduction to dominant}. 
\begin{proof}[Proof of Proposition \ref{reduction to dominant}]
Let $\varphi:X\rightarrow V$ be a strongly dominant $K$-morphism,
where $K$ is a number field. We may assume that $K/\rats$ is Galois.
By restriction of scalars we obtain a $\rats$-morphism $\mathrm{Res}_{\rats}^{K}\varphi:\mathrm{Res}_{\rats}^{K}(X)\rightarrow\mathrm{Res}_{\rats}^{K}(V)$.
By our assumption, $\mathrm{Res}_{\rats}^{K}(\varphi)^{n}=\mathrm{Res}_{\rats}^{K}(\varphi^{n})$
is (FRS) for some $n\in\nats$. Since the (FRS) property is preserved
under base change to $K$, we obtain that 
\[
\left((\mathrm{Res}_{\rats}^{K}\varphi)^{n}\right)_{K}=\left((\mathrm{Res}_{\rats}^{K}\varphi)_{K}\right)^{n}:\left(\mathrm{Res}_{\rats}^{K}(X)_{K}\right)^{n}\longrightarrow\left(\mathrm{Res}_{\rats}^{K}(V)_{K}\right)^{n}
\]
is (FRS). Finally, notice that $\mathrm{Res}_{\rats}^{K}(X)_{K}$
(resp. $\mathrm{Res}_{\rats}^{K}(V)_{K}$) is a product of $l:=[K:\rats]$
$K$-varieties $X_{1},{\ldots},X_{l}$ (resp. $K$-vector spaces $V_{1},{\ldots},V_{l}$),
each of which is $\mathbb{Q}$-isomorphic to $X$ (resp. $V$), and
that $\mathrm{Res}_{\rats}^{K}(\varphi)_{K}^{n}$ can be written as,
\[
\varphi_{1}^{n}\times\dots\times\varphi_{l}^{n}:X_{1}^{n}\times{\ldots}X_{l}^{n}\rightarrow V_{1}\times\dots\times V_{l},
\]
where each $\varphi_{i}$ is a twist of $\varphi_{K}$ by some Galois
element $\sigma_{i}\in\mathrm{Gal}(K/\rats)$, with $\varphi_{1}:=\varphi$.
This allows us to deduce that for each $i$, the morphism $\varphi_{i}^{n}$
is (FRS), and in particular $\varphi^{n}$ is (FRS) as well. 
\end{proof}

\section{Convolution properties of algebraic families of morphisms - a uniform
version of Theorem \ref{Main result}}

\label{sec:Convolution-properties-of} Let $K$ be a field of characteristic
$0$. Our goal in this section is to prove Theorem \ref{generalization of the main result}: 
\begin{thm}[Theorem \ref{generalization of the main result}]
\label{generalization of main theorem -7.1} Let $Y$ be a $K$-variety,
set $V=\mathbb{A}_{K}^{m}$, let $\widetilde{X}$ be a family of varieties
over $Y$ and let $\widetilde{\varphi}:\widetilde{X}\rightarrow V\times Y$
be a $Y$-morphism. Then, 
\begin{enumerate}
\item The set $Y':=\{y\in Y:\widetilde{X}_{y}\text{ is smooth and }\widetilde{\varphi}_{y}:\widetilde{X}_{y}\rightarrow V\text{ is strongly dominant}\}$
is constructible.
\item There exists $N\in\nats$ such that for any $n>N$, and any $n$ points
$y_{1},\ldots,y_{n}\in Y'(K)$, the morphism $\widetilde{\varphi}_{y_{1}}*\dots*\widetilde{\varphi}_{y_{n}}:\widetilde{X}_{y_{1}}\times\dots\times\widetilde{X}_{y_{n}}\rightarrow V$
is (FRS). 
\end{enumerate}
\end{thm}

We first reduce to a similar statement about self convolutions: 
\begin{thm}
\label{generalization-self convolutions} Let $Y,V,\widetilde{X},\widetilde{\varphi}$
and $Y'$ be as in Theorem \ref{generalization of main theorem -7.1}.
Then $Y'$ is constructible and there exists $N\in\nats$ such that
for any $n\geq N$ and any $y\in Y'(K)$, the morphism $\varphi_{y}^{n}:\widetilde{X}_{y}^{n}\rightarrow V$
is (FRS). 
\end{thm}

\begin{lem}
\label{lem7.3:Theorem--implies} Theorem \ref{generalization-self convolutions}
implies Theorem \ref{generalization of main theorem -7.1}. 
\end{lem}

\begin{proof}
Let $\widetilde{\varphi},Y,V,Y'$ and $N\in\nats$ be as in Theorem
\ref{generalization-self convolutions}, and let $y_{1},{\ldots},y_{2N}\in Y'(K)$.
Notice that for any $1\leq i\leq2N$ we have that $\widetilde{\varphi}_{y_{i}}^{N}$
is (FRS). For each $i$, let $\mu_{i}$ be a non-negative Schwartz measure on $\widetilde{X}_{y_{i}}(F)$ 
for a non-Archimedean field $K\subseteq F\in\mathrm{Loc}$,  supported on $y_i$, such that $N$-th convolution
power $(\widetilde{\varphi}_{y_{i}})_{*}(\mu_{i})^{N}$ has continuous
density. 
By an argument similar to the one in Proposition \ref{prop:reduction to (FRS) at diagonal},
we can deduce that $\mathcal{F}(\widetilde{\varphi}_{y_{i}*}(\mu_{i}))$
has an ${L}^{2N}$-density and that $(\widetilde{\varphi}_{y_{1}})_{*}(\mu_{1})*\dots*(\widetilde{\varphi}_{y_{2N}})_{*}(\mu_{2N})$
has continuous density. As a consequence, we get that $\widetilde{\varphi}_{y_{1}}*\dots*\widetilde{\varphi}_{y_{2N}}$
is (FRS). 
\end{proof}
We now wish to prove Theorem \ref{generalization-self convolutions}.
We start by proving Lemmas \ref{lemma 7.4} and \ref{lemma 7.5} in
order to deduce that $Y'$ is constructible.
\begin{lem}
\label{lemma 7.4}Let $\widetilde{X}$ and
$Y$ be as in Theorem \ref{generalization of main theorem -7.1}.
Then the set $\hat{Y}:=\{y\in Y:\widetilde{X}_{y}\text{ is smooth}\}$
is constructible. 
\end{lem}

\begin{proof}
Let $\pi_{\widetilde{X}}:\widetilde{X}\rightarrow Y$
be the structure morphism. By Chevalley's Theorem, it is enough to
show that the set $\pi_{\widetilde{X}}^{-1}(\hat{Y})=\{x\in\widetilde{X}:\widetilde{X}_{\pi_{\widetilde{X}}(x)}\text{ is smooth}\}$
is constructible. Consider the following functions {from $X$ to $\ints$}: 
\begin{enumerate}
\item$x\mapsto\mathrm{\mathrm{dim}}_{x}\widetilde{X}_{\pi_{\widetilde{X}}(x)}:=\max\limits _{X_{i}\in\mathrm{Irr}(\widetilde{X}_{\pi_{\widetilde{X}}(x)})}\curly{\dim X_{i}:x\in X_{i}}$. 
\item $x\mapsto\mathrm{\mathrm{dim}}T_{x}\widetilde{X}_{\pi_{\widetilde{X}}(x)}$,
assigning to each $x\in\widetilde{X}$ the dimension of the Zariski
tangent space of $\widetilde{X}_{\pi_{\widetilde{X}}(x)}$ at $x$.
\end{enumerate}
Let $\hat{X}:=\{x\in\widetilde{X}:\mathrm{\mathrm{dim}}T_{x}\widetilde{X}_{\pi_{\widetilde{X}}(x)}\neq\mathrm{\mathrm{dim}}_{x}\widetilde{X}_{\pi_{\widetilde{X}}(x)}\}$
and notice that $\pi_{\widetilde{X}}(\hat{X})=\{y\in Y:\,\widetilde{X}_{y}\text{ is not smooth}\}$
and that $\hat{Y}$ is the complement of $\pi_{\widetilde{X}}(\hat{X})$ in $Y$. By a
corollary of Chevalley's theorem, the function $\mathrm{dim}_{x}\widetilde{X}_{\pi_{\widetilde{X}}(x)}$
is upper semi-continuous. Since $\mathrm{\mathrm{dim}}T_{x}\widetilde{X}_{\pi_{\widetilde{X}}(x)}$
is also upper semi-continuous, it follows that $\hat{X}$
is a constructible set. This implies that $\hat{Y}$ is constructible
\end{proof}
\begin{lem}
\label{lemma 7.5} Let $S$ be a finite type $K$-scheme,
$Z$ be an absolutely irreducible finite type $K$-scheme and $\varphi:X\rightarrow Z\times S$
be an $S$-morphism. Then 
\[
\widetilde{S}:=\{s\in S\text{ such that }\varphi_{s}\text{ is strongly dominant}\}
\]
is a constructible subset of $S$. 
\end{lem}

\begin{proof}
We may assume that $S$ is irreducible. Let $\pi_{X}:X\rightarrow S$
and $\pi_{Z\times S}:Z\times S\rightarrow S$ be
the structure morphisms, and $\eta$ be the generic point of $S$. 
In order to prove the lemma, it is enough to prove that $\eta\in\widetilde{S}$
(resp. $\eta\notin\widetilde{S}$) implies the existence of an open
set $U$, such that $U\subseteq\widetilde{S}$ (resp. $U\cap\widetilde{S}=\varnothing$).
By showing this we deduce the lemma by Noetherian induction. 

By \cite[Lemma 36.22.8]{stacks-project}, we may assume that we have an affine
irreducible scheme $S'$, a surjective finite \'{e}tale morphism $\psi:S'\rightarrow S$,
and a fibered diagram 
\[
\begin{array}{ccc}
X_{S'} & \longrightarrow & X\\
\downarrow\pi_{X_{S'}} & \, & \downarrow\pi_{X}\\
S' & \overset{\psi}{\longrightarrow} & S
\end{array}
\]
such that all irreducible components of the generic fiber of $\pi_{X_{S'}}$
are absolutely irreducible. Denote by $\widetilde{\varphi}:X_{S'}\rightarrow Z \times S'$
the corresponding base change of $\varphi$ to $S'$. Since $\psi$
is finite, we have for any $s'\in S'$
that $\widetilde{\varphi}_{s'}$ is a base change of $\varphi_{\psi(s')}$
by a finite field extension, and hence $\widetilde{\varphi}_{s'}$
is strongly dominant if and only if $\varphi_{\psi(s')}$ is strongly
dominant. Since $\psi$ is \'{e}tale, it is an open
map, so we may assume that $S=S'$ and that all irreducible components
of $X_{\eta}:=\pi_{X}^{-1}(\eta)$ are absolutely irreducible.

By \cite[Proposition  9.7.8, 9.7.12]{Gro66} (or \cite[Theorem 3.43]{AA18}),
there exists an open set $U\subseteq S$ such that for any $s\in U$
the following holds: 
\begin{enumerate}
\item There is a bijection $\Phi_{\eta,s}$ between the number of irreducible
components $X_{\eta,1},\ldots,X_{\eta,t}$ of $X_{\eta}$ and the
number of irreducible components of $X_{s}$. 
\item All irreducible components of $X_{s}$ are absolutely irreducible. 
\end{enumerate}
The bijection $\Phi_{\eta,s}$ is constructed as follows. For each
irreducible component $X_{\eta,j}$ of $X_{\eta}$, we have that $\overline{X}_{\eta,j}\cap X_{s}$
is an irreducible component of $X_{s}$. Under this bijection, we
deduce that the set 
\[
\widetilde{X}_{j}:=\{x\in\pi_{X}^{-1}(U):x\in j\text{-th irreducible component of }X_{\pi_{X}(x)}\}
\]
is locally closed. For $s\in U$, the condition
that $\varphi_{s}$ is strongly dominant is equivalent to $\left(\varphi_{j}\right)_{s}:=\left(\varphi|_{\widetilde{X}_{j}}\right)_{s}:(\widetilde{X}_{j})_{s}\rightarrow Z\times_{\spec(K)}\spec(K(\curly{s}))$
being dominant for any $1\leq j\leq t$. Hence, it is enough to show
that the set 
\[
\widetilde{U}_{j}:=\{s\in U\text{ such that }\left(\varphi_{j}\right)_{s}\text{ is dominant}\},
\]
is constructible. Notice that $\left(\varphi_{j}\right)_{\eta}$ is
dominant if and only if $\varphi_{j}(\widetilde{X}_{j})_{\eta}=\left(\varphi_{j}\right)_{\eta}((\widetilde{X}_{j})_{\eta})$
is dense in $Z\times_{\spec(K)}\spec(K(\curly{\eta}))$.
It now follows from \cite[Lemma 36.22.3 and 36.22.4]{stacks-project} that
we may find an open set $W\subseteq U$ such that, 
\begin{enumerate}
\item If $\left(\varphi_{j}\right)_{\eta}$ is dominant then $\left(\varphi_{j}\right)_{s}$
is dominant for any $s\in W$. 
\item If $\left(\varphi_{j}\right)_{\eta}$ is not dominant then $\left(\varphi_{j}\right)_{s}$
is not dominant for any $s\in W$. 
\end{enumerate}
By Noetherian induction we deduce that $\widetilde{U}_{j}$ is constructible
for any $j$ and thus we are done. 
\end{proof}
Combining the above lemmas, we obtain that the set 
\[
Y':=\{y\in Y:\widetilde{X}_{y}\text{ is smooth, and }\widetilde{\varphi}_{y}:\widetilde{X}_{y}\rightarrow V\text{ is strongly dominant}\}
\]
is constructible. To complete the proof of Theorem \ref{generalization-self convolutions}
we need the following lemma: 
\begin{lem}
\label{Lem-generic fiber}Let $\varphi:X\rightarrow Y$ be a morphism
of irreducible $K$-varieties and let $U\subseteq X$ be an open set
that contains the generic fiber of $\varphi$. Then there exists an
open set $W\subseteq Y$ such that $\varphi^{-1}(W)\subseteq U$. 
\end{lem}

\begin{proof}
We may assume that $X$ and $Y$ are affine. Let $\eta_{Y}$ be the
generic point of $Y$ and $\varphi^{*}:K[Y]\rightarrow K[X]$ be the
ring map corresponding to $\varphi$. The generic fiber $\varphi^{-1}(\eta_{Y})$
can be written as 
\[
\varphi^{-1}(\eta_{Y})=\{[\mathfrak{p}]\in\spec(K[X]):(\varphi^{*})^{-1}(\mathfrak{p})=(0)\}.
\]
Let $U\subseteq X$ be an open set containing $\varphi^{-1}(\eta_{Y})$,
then we can write 
\[U=D(I):=\curly{[\frak{p}]\in\spec(K[X]):I\not\subseteq\frak{p}}.\] 
By Noetherianity, $I$ is an intersection of finitely many prime ideals,
so it is enough to prove the lemma for the case where $I$ is prime.
Notice that the condition $D(I)\supseteq\varphi^{-1}(\eta_{Y})$ implies
that $(\varphi^{*})^{-1}(I)\neq\{0\}$. Hence, there exists some $f\in K[Y]$
such that $\varphi^{*}(f)=f\circ\varphi\in I$. But then, we have
$D(f\circ\varphi)\subset U$, and thus ${D}(f\circ\varphi)=\varphi^{-1}({D(f)})$,
so we are done. 
\end{proof}
We can now finish the proof of Theorem \ref{generalization-self convolutions}: 
\begin{proof}[Proof of Theorem \ref{generalization-self convolutions}]
Since $Y'$ is a constructible subset of $Y$, we may write $Y'=\bigcup\limits _{i=1}^{l}Y_{i}$,
where each $Y_{i}$ is a locally closed subset of $Y$. It is enough
to prove the theorem for each $Y_{i}\subseteq Y'$. By decomposing
$Y_{i}$ into irreducible components $Y_{i,1},{\ldots},Y_{i,d}$,
it is enough to prove the theorem for the case where $Y=Y'$ and $Y$ is irreducible.
Note that, for any $y\in Y$, the fiber $\widetilde{X}_{y}$
is smooth and $\widetilde{\varphi}_{y}:\widetilde{X}_{y}\rightarrow V$
is a strongly dominant morphism. 

Denote by $\eta_{Y}$ the generic point of $Y$. We first want to
show the existence of an open set $W\subseteq Y$ and $N_{W}\in\nats$
such that $\widetilde{\varphi}_{y}^{N_{W}}:\widetilde{X}_{y}^{N_{W}}\rightarrow V$
is (FRS) for all $y\in W$. By generic flatness (\cite[Theorem 6.9.1]{Gro66}),
we may assume that $\widetilde{X}$ is flat over $Y$. Denote by $\pi_{\widetilde{X}}:\widetilde{X}\rightarrow Y$
its structure map. By Theorem \ref{Main result}, there exists $N(\eta)\in\nats$
such that the $N(\eta)$-th convolution power $\widetilde{\varphi}_{\eta}^{N(\eta)}$
is (FRS). Denote $\widetilde{X}^{N(\eta),Y}:=\widetilde{X}\times_{Y}\dots\times_{Y}\widetilde{X}$
and let $\pi_{\widetilde{X}}^{N(\eta)}:\widetilde{X}^{N(\eta),Y}\rightarrow Y$
be the corresponding (flat) structure map. By \cite[Corollary 2.2]{AA16},
the set 
\[
U:=\{x\in\widetilde{X}^{N(\eta),Y}(\overline{K}):\widetilde{\varphi}_{\pi_{\widetilde{X}}^{N(\eta)}(x)}^{N(\eta)}:\widetilde{X}_{\pi_{\widetilde{X}}^{N(\eta)}(x)}^{N(\eta)}\rightarrow V\text{ is (FRS) at }x\}
\]
is open. Note that we used here the flatness of $\pi_{\widetilde{X}}^{N(\eta)}$.
Since $U$ contains the generic fiber $\left(\pi_{\widetilde{X}}^{N(\eta)}\right)^{-1}(\eta)$,
by Lemma \ref{Lem-generic fiber} we have $U\supseteq\left(\pi_{\widetilde{X}}^{N(\eta)}\right)^{-1}(W)$
for some open $W\subseteq Y$, and hence $\widetilde{\varphi}_{y}^{N(\eta)}:\widetilde{X}_{y}^{N(\eta),Y}\rightarrow V$
is (FRS) for any $y\in W$. We can now repeat the process for the
closed subvariety $Y_{1}:=Y\backslash W$. By Noetherian induction,
we can deduce the existence of $N\in\nats$ such that $\varphi_{y}^{N}:\widetilde{X}_{y}^{N}\rightarrow V$
is (FRS) for any $y\in Y$ as required. 
\end{proof}
We are now ready to prove Corollary \ref{Generalization for varieties of bounded complexity}. 
We first introduce the notion of complexity. 
\begin{defn}
Denote by $\mathcal{C}_{D}$ the class of all $K$-schemes and $K$-morphisms of complexity at most $D$, which is defined in the following way:
\begin{enumerate}
\item {An affine $K$-scheme $X$ has }\textit{{complexity
at most }}{$D$ (i.e $X\in\mathcal{C}_{D}$) if $X$
has a closed embedding $\psi_{X}:X\hookrightarrow\mathbb{A}_{K}^{N}$
with $K[X]=K[x_{1},{\ldots},x_{N}]/\langle f_{1},{\ldots},f_{k}\rangle$,
where $N,k$ and $\max\limits _{i}\{\deg(f_{i})\}$ are at most $D$. }
\item {A morphism $\varphi:X\rightarrow Y$ of affine $K$-schemes
is said to be of }\textit{{complexity at most }}{$D$
(i.e $\varphi\in\mathcal{C}_{D}$), if $X,Y\in\mathcal{C}_{D}$ with
embeddings $\psi_{X}:X\hookrightarrow\mathbb{A}_{K}^{N_{1}}$ and
$\psi_{Y}:Y\hookrightarrow\mathbb{A}_{K}^{N_{2}}$ as above, such
that the polynomial map $\widetilde{\varphi}:\mathbb{A}_{K}^{N_{1}}\rightarrow\mathbb{A}_{K}^{N_{2}}$
induced from $\varphi$ is of the form $\widetilde{\varphi}=(\widetilde{\varphi}_{1},{\ldots},\widetilde{\varphi}_{N_{2}})$,
where the degrees of $\{\widetilde{\varphi}_{i}\}_{i=1}^{N_2}$ are at most $D$. }
\item {A separated $K$-scheme $X$ is in $\mathcal{C}_{D}$
if it has an open affine cover $X=\bigcup_{i=1}^{N}U_{i}$ where $U_{1},\ldots,U_{N}\in\mathcal{C}_{D}$,
such that $N\leq D$, and the transition maps of the (affine) intersections
$\{U_{i}\cap U_{j}\}_{i,j}$ are in $\mathcal{C}_{D}$.}
\item {A morphism $\varphi:X\rightarrow Y$ of separated
$K$-schemes is in $\mathcal{C}_{D}$, if there exist open affine
covers $Y=\bigcup\limits _{j=1}^{N}V_{j}$, and $X=\bigcup\limits _{j=1}^{N}\bigcup\limits _{i=1}^{N_{j}}U_{ij}$
with $N,N_{j}\leq D$ for any $j$, and such that the following holds
for any $1\leq j,j'\leq N$ and $1\leq i,i'\leq N_{j}$:}
\begin{enumerate}
\item {$\varphi(U_{ij})\subseteq V_{j}$.}
\item {The morphisms $\varphi|_{U_{ij}}:U_{ij}\rightarrow V_{j}$
are in $\mathcal{C}_{D}$.}
\item {The transition maps of the intersections $\{U_{i'j'}\cap U_{ij}\}_{i,i',j,j'}$
are in $\mathcal{C}_{D}$.}
\item {The transition maps of the intersections $\{V_{j'}\cap V_{j}\}_{j,j'}$
are in $\mathcal{C}_{D}$.}
\end{enumerate}
\item {A $K$-scheme $X$ is in $\mathcal{C}_{D}$ if it
has an open affine cover $X=\bigcup_{i=1}^{N}U_{i}$ by $U_{1},\ldots,U_{N}\in\mathcal{C}_{D}$,
such that $N\leq D$, and the transition maps of the intersections
$\{U_{i}\cap U_{j}\}$ are in $\mathcal{C}_{D}$ (note that the schemes
$\{U_{i}\cap U_{j}\}$ are not necessarily affine but are separated).
We say that a morphism $\varphi:X\rightarrow Y$ of $K$-schemes is
in $\mathcal{C}_{D}$ if it satisfies the same demands as in ($4)$. }
\end{enumerate}
\end{defn}

Notice that the set $Z_{D}$ of all affine $K$-schemes
$X\in\mathcal{C}_{D}$ is an algebraic family. Indeed, let $X\in Z_{D}$
be an affine $K$-scheme with coordinate ring $K[X]=K[x_{1},{\ldots},x_{m}]/\langle f_{1},{\ldots},f_{k}\rangle$
consisting of $m$ generators and $k$ relations. Since the degrees
and the number of variables of $\{f_{i}\}_{i=1}^{k}$ are bounded
by $D$, the set of all such schemes has a natural structure of an
affine space $Z_{m,k}$. We can write $Z_{D}$ as a disjoint union
of affine spaces $Z_{D}=\bigsqcup\limits _{m,k=1}^{D}Z_{m,k}$. 

Let $V=\mathbb{A}_{K}^{m}$ and $D>m$. By a similar
argument, the set of all morphisms $\varphi:W\rightarrow\mathbb{A}_{K}^{m}$
in $\mathcal{C}_{D}$ forms an algebraic family, denoted $Y_{D}$.
We can now prove Corollary \ref{Generalization for varieties of bounded complexity}: 
\begin{cor}
\label{convolution complexity}For any $m<D\in\nats$, there exists
$N(D)\in\nats$ such that for any $n>N(D)$ and $n$ strongly dominant
morphisms $\{\varphi_{i}:X_{i}\rightarrow\mathbb{A}_{K}^{m}\}_{i=1}^{n}$
of complexity at most $D$, with $\curly{X_{i}}_{i=1}^{n}$ smooth
$K$-varieties, the morphism $\varphi_{1}*\dots*\varphi_{n}$ is (FRS). 
\end{cor}

\begin{proof}
Consider the variety $Y_{D}$ as constructed above. Let $\widetilde{X}:=\{(w,(W,\varphi)):w\in W,(W,\varphi)\in Y_{D}\}$
and define $\widetilde{\varphi}:\widetilde{X}\rightarrow\mathbb{A}_{K}^{m}\times Y_{D}$
by 
\[
\widetilde{\varphi}(w,(W,\varphi))=\left(\varphi(w),(W,\varphi)\right).
\]
Note that $\widetilde{X}$ is a $Y_{D}$-variety, $\widetilde{\varphi}$
is a $Y_{D}$-morphism and that for any $(W,\varphi)\in Y_{D}$, the
fiber $\widetilde{\varphi}_{(W,\varphi)}$ is just the map $\varphi:W\rightarrow\mathbb{A}_{K}^{m}$.
The corollary now follows from Theorem \ref{generalization of main theorem -7.1}. 
\end{proof}
\bibliographystyle{alpha}
\bibliography{bibfile}
\end{document}